\documentclass[a4paper,11pt]{article}
\usepackage{epsf,epsfig}
\usepackage{graphics,color}
\usepackage{amsmath}
\usepackage{amssymb}
\usepackage{cite}
\usepackage{verbatim}
\usepackage{float}
\usepackage{graphicx}
\usepackage{amsthm}
\usepackage{textcomp}
\usepackage{subfig}
\usepackage{hyperref}
\usepackage{yhmath}
\usepackage{mathrsfs}
\newif\ifdraft
\draftfalse 

\numberwithin{equation}{section}
\newtheorem{theorem}{Theorem}[section]
\newtheorem{prop}[theorem]{Proposition}
\newtheorem{cor}[theorem]{Corollary}
\newtheorem{lemma}[theorem]{Lemma}

\newtheorem{remark}[theorem]{Remark}

\newtheorem{definition}[theorem]{Definition}

\newcommand{\jump}[1]{\text{{\rm \textlbrackdbl}}{#1}\text{{\rm \textrbrackdbl}}}
\newcommand{\Om}{\Omega}

\def\area{\mathcal A}

\def\cof{\,{\rm cof\ }}

\def\curl{\,{\rm Curl\ }}
\def\det{\,{\rm det\ }}

\def\dist{\,{\rm dist}}
\def\div{\,{\rm div\ }}

\newcommand{\R}{\mathbb{R}}

\newcommand\Suno{{\mathbb S}^1}

\newcommand{\res}      {\mathop{\hbox{\vrule height 7pt width .5pt depth
                        0pt\vrule height .5pt width 6pt depth 0pt}}\nolimits}

\title{On the relaxation of polyconvex functionals with linear growth under strict convergence in $BV$
}

\author{
	Riccardo Scala\footnote{ 
		Dipartimento di Ingegneria dell'Informazione e Scienze Matematiche, Universit\`a di Siena, 53100 Siena, Italy.
		E-mail: riccardo.scala@unisi.it}
}

\begin{document}

\maketitle

\begin{abstract}
We consider the relaxation of polyconvex functionals with linear growth with respect to the strict convergence in the space of functions of bounded variation. These functionals appears as relaxation of  $F(u,\Omega):=\int_\Omega f(\nabla u)dx$, where $u:\Om\rightarrow \R^m$, and $f$ is polyconvex. In constrast with the case of relaxation with respect to the standard $L^1$-convergence, in the case that $\Omega$ is $2$-dimensional, we prove that the sets map $A\mapsto F(u,A)$ for $A$ open, is, for every $u\in BV(\Om;\R^m)$, $m\geq1$, the restriction of a Borel measure. This is not true in the case $\Omega\subset\R^n$, with $n\geq3$. Using the integral representation formula for a special class of functions, we also show the presence of Cartesian maps whose relaxed area functional with respect to the $L^1$-convergence is strictly larger than the area of its graph. 
\end{abstract}


\vspace{2mm}

\noindent {\bf AMS (MOS) 2020 Subject Clas\-si\-fi\-ca\-tion:}  49J45, 49Q05, 49Q15, 28A05, 28A75, 74B20.

\section{Introduction}
Polyconvexity arises in non-linear elasticity as in many branches of mechanics of solids, and is a more realistic hypothesis on the energy functional than just convexity \cite{Ball}.
The setting under consideration in this paper is the one where the growth of the involved functional is linear, circumstance in which the standard lower semicontinuity results  \cite{GMS3,FuHu:95} do not apply. 

Given an open bounded set $\Omega$, the prototype example of energy with this growth condition is provided by the area functional that, given a map $u:\Om\subset\R^n\rightarrow \R^m$ smooth enough, computes the $n$-dimensional Hausdorff measure of the graph $G_u:=\{(x,y)\in \Om\times \R^m:y=u(x)\}$ of $u$. Thanks to the area formula, the area functional takes the form
\begin{align}
	\mathbb A(u,\Om):=\int_\Om|\mathcal M(\nabla u)|dx,
\end{align}  
where $\mathcal M(\nabla u)$ is the vector whose entries are all the determinants of the $k\times k$-submatrices of $\nabla u$, $k=0,\dots,\min\{n,m\}$ (the $0\times0$ determinant is conventionally taken as $1$). More generally, we consider energies such as 
\begin{align}\label{functionalF}
F(u,\Om)=\int_\Om f(\nabla u)dx,
\end{align}
where $f$ is polyconvex, that is, there exists a convex function $g$ such that 
\begin{align}\label{functionalF2}
	f(\nabla u)=g(\mathcal M(\nabla u)).
\end{align}
The condition of linear growth considered in \cite{ADM} is expressed by the relation
\begin{align}\label{growth_g}
g(\mathcal M(\nabla u))\geq c_0|\mathcal M(\nabla u)|,
\end{align}
for some positive constant $c_0$. 
Due to the lack of lower semicontinuity of this  kind of functionals a relaxation procedure is necessary. This approach has been studied  in \cite{ADM}, where the authors considered the $L^1$-relaxation of $F$ given by
\begin{align}\label{rel_L1}
\mathcal F^{L^1}(u,\Om)=\inf\{\liminf_{k\rightarrow\infty}F(u_k,\Om):(u_k)\subset C^1(\Om;\R^m),\;u_k\rightarrow u\text{ in }L^1(\Om;\R^m)\},
\end{align}
and defined for any $u\in L^1(\Om;\R^m)$. The relaxed functional $\mathcal F^{L^1}$ turns out to be $L^1$-lower semicontinuous and extend the functional $F$ from $C^1(\Om;\R^m)$ to $L^1(\Om;\R^m)$. However,  the behaviour of $\mathcal F^{L^1}$ is extremely wild, due to non-local phenomena that arise already for the relaxed area functional as soon as $n,m>1$. 
 Apart from the $1$-dimensional case ($n=1$) that is much simpler, assuming $n\geq2$, there is a big difference between the one codimensional case ($m=1$) and the higher codimensional one. Indeed, if $u$ is scalar valued, then the functional $\mathcal F^{L^1}$ is local and admits an integral representation: In the special case of the relaxed area functional, which we denote by $\mathcal A^{L^1}$, it can be proved that the domain of $  \mathcal A^{L^1}$ is the space $BV(\Om)$ and that
 \begin{align}
\mathcal A^{L^1}(u,\Om)=\int_\Om\sqrt{1+|\nabla u|^2}dx+|D^su|(\Om), \qquad \qquad \forall u\in BV(\Om),
 \end{align}
where $\nabla u$ denotes the approximate gradient of $u$ and $D^su$ the singular part of the distributional derivative $Du$ of $u$. A similar expression in terms of the recession function of $F$ holds in the  case of general function $g$ (see \cite{DalMaso:80}).

Instead, the case $m\geq2$ does not enjoy so good properties: For general $u\in BV(\Om;\R^m)$ it can be proved only that 
 \begin{align}\label{1.6}
	\mathcal A^{L^1}(u,\Om)\geq \int_\Om\sqrt{1+|\nabla u|^2}dx+|D^su|(\Om), \qquad \qquad \forall u\in BV(\Om;\R^m),\nonumber\\
		\mathcal F^{L^1}(u,\Om)\geq \int_\Om g(\mathcal M(\nabla u))dx+c_0|D^su|(\Om), \qquad \qquad \forall u\in BV(\Om;\R^m),
\end{align}
and that there exist maps of bounded variations for which $\mathcal A^{L^1}$ (and $\mathcal F^{L^1}$) is $+\infty$ (see \cite{BePaTe:15,BePaTe:16}). The domain of $\mathcal A^{L^1}$, namely the set of maps for which $\mathcal A^{L^1}$ is finite, is a subset of $BV(\Om;\R^m)$, whose precise description is not available. Moreover, it has been proved in \cite{ADM} that, for a fixed $u\in BV(\Om;\R^m)$, the set function $A\subset\Om\rightarrow \mathcal F^{L^1}(u,A)$  is not in subadditive, and thus $\mathcal F^{L^1}$ does not admit any integral representation. This is true also for the area functional, where the non-subadditivity property has been incountered already for two simple examples of functions: The vortex map $u_V$ and the triple junction function $u_T$. The former is the Sobolev map $u_V(x)=\frac{x}{|x|}$ in the ball $\Om=B_R(0)\subset\R^2$, the latter $u_T:B_R(0)\subset\R^2\rightarrow \{\alpha,\beta,\gamma\}$ is a piecewise constant map assuming three values that are the three vertices of an equilateral triangle in $\R^2$. For both these functions,  suggested by De Giorgi in \cite{DeGiorgi:92}, Acerbi and Dal Maso proved the non-subadditivity property exploiting suitable lower and upper bounds for $\mathcal A^{L^1}$. Also, the precise values of $\mathcal A^{L^1}(u_V,B_R(0))$ and $\mathcal A^{L^1}(u_T,B_R(0))$ were not available at that time, and only recently it has been possible to find them explicitely (see \cite{BePa:10,Scala:19,BES1,BES2,BES3}). In the last references, it is clear how the nonlocality of $\mathcal A^{L^1}(u_V,\cdot)$ and $\mathcal A^{L^1}(u_T,\cdot)$ pops up: 
In the former case, we have 
\begin{align}
\mathcal A^{L^1}(u_V,B_R(0))= \int_{B_R(0)}\sqrt{1+|\nabla u_V|^2}dx+\mathcal H^2(C_R),
\end{align}
where $\mathcal H^2(C_R)$ is the $2$-dimensional Hausdorff measure of a minimal surface $C_R$ obtained by solving a particular non-parametric Plateau problem with partial free boundary in codimension $1$. This object, whose shape is (the half of) a sort of catenoid constrained to contain a segment, is a suitable projection in $\R^3$ of the vertical part of the cartesian current $S$ obtained as limit of the graphs $ G_{u_k}$ of a recovery sequence $(u_k)\subset C^1(B_R;\R^2)$ for $  \mathcal A^{L^1}(u_V,B_R(0))$ (see \cite{BES2} for the non-parametric Plateau problem and \cite{BES1,BES3} for the computation of $\mathcal A^{L^1}(u_V,B_R(0))$). The radius $R>0$ represents  the height of the catenoid, and hence the area of $C_R$ depends on $R$, in such a way that 
$\mathcal H^2(C_R)\leq 2\pi R$; for $R$ larger than a certain threshold it happens that $\mathcal H^2(C_R)=\pi$.
A similar phenomenon is observed for $u_T$, where the singular contribution in $\mathcal A^{L^1}(u_T,B_R(0))$ is provided by the area of three minimal surfaces in $\R^3$ solving a nonparametric Plateau-type problem with partial free boundary. Also in this case, these minimal surfaces have the role of filling the holes in the graph of $G_{u_T}$, hence arising as vertical parts of the cartesian current obtained as limit of the graphs $G_{u_k}$ of a recovery sequence $(u_k)\subset C^1(B_R;\R^2)$ for $  \mathcal A^{L^1}(u_T,B_R(0))$ (see \cite{BePa:10,Scala:19}).

The relaxed area of $u_V$ and $u_T$ in a ball $B_R(0)$ are the unique non-trivial cases in which $\mathcal A^{L^1}(u,\Om)$ is explicit, and minimal changes in the geometry of the domain or on the choice of the function $u$ makes the computation of $\mathcal A^{L^1}(u,\Om)$ out of reach; in more general cases, only (non-sharp) upper bounds are available, as in \cite{BSS} for the case of Sobolev maps with values in $\mathbb S^1$ (thus generalizing the vortex map) and in \cite{BeElPaSc:19,ScSc} for the case of piecewise constant functions taking three values (hence generalizing the triple junction function). In any case, we believe that the vertical parts of cartesian currents obtained as limits of the graphs $G_{u_k}$ of a recovery sequence $(u_k)\subset C^1(B_R;\R^2)$ can be often described, in a similar fashion as for $u_T$ and $u_V$, as minimal surfaces  arising as solutions of non-parametric Plateau problems with partial free boundaries (see \cite{BMS}) or semicartesian Plateau problems (see \cite{BePaTe:15,BePaTe:16}). 

One of the issue encontered in the analysis of the relaxation in \eqref{rel_L1} is that, when one considers, for $u\in BV(\Om;\R^m)$, a sequence $(u_k)\subset C^1(\Om;\R^m)$ realizing the infimum (i.e., a so-called recovery sequence), then the limit of the graphs $G_{u_k}$ in $\Om\times \R^m$, seen as integral currents, cannot be easily identified. Indeed, it is only known that 
$$G_{u_k}\rightharpoonup G_u+V_{{\rm min}}=:S_{{\rm min}},$$
where $V_{{\rm min}}$ is called vertical part, and is such that $\partial V_{{\rm min}}=-\partial G_u$. But unless few general properties on $V_{{\rm min}}$ (that are common to vertical parts of cartesian currents, see \cite{GMS2}) nothing can be said, a priori,  on its geometry. The knowledge of $V_{{\rm min}}$ would give rise, at least for the area functional, the trivial lower bound (which follows by lower-semicontinuity of the mass)
$$\mathcal A^{L^1}(u,\Om)\geq |S_{{\rm min}}|=|G_u|+|V_{{\rm min}}|,$$
where by $|\cdot|$ we indicate the total mass of a current.
However, $V_{{\rm min}}$ strongly depends on $\Om$, in general, and this is the main reason of non-locality of $\mathcal A^{L^1}$ (and of $\mathcal F^{L^1}$). 

In contrast, this phenomenon disappears, at least in the case $n=2$, if one consider the relaxation of $F$ with respect to strict topology in $BV(\Om;\R^m)$. Namely, let us consider, for $\Om\subset\R^2$ and for all $u\in BV(\Om;\R^m)$, the functional 
\begin{align}\label{rel_strict}
\mathcal F(u,\Om)=\inf\{\liminf_{k\rightarrow\infty}F(u_k,\Om):(u_k)\subset C^1(\Om;\R^m),\;u_k\rightarrow u\text{ strictly in  }BV(\Om;\R^m)\}.
\end{align}
It is then possible to show that if $u_k\subset C^1(\Om;\R^m)$ converges to $u$ strictly in $BV(\Om;\R^m)$ and $\mathbb A(u_k,\Om)<C<+\infty$ for all $k$, then
\begin{align}\label{19}
G_{u_k}\rightharpoonup G_u+V_{{\rm strict}}=:S_{{\rm strict}} \qquad \qquad \text{ as currents,} 
\end{align}
where $V_{{\rm strict}}$ (and hence $S_{{\rm strict}}$) is uniquely determined and does not depend on the specific sequence $u_k$. This result has been proved in \cite{Mucci}, where relaxation in \eqref{rel_strict} has been considered for the area functional. The relaxed area functional under strict convergence has been analyzed more in detail in \cite{BCS,BCS2,C,CM}. 
 Due to the more restrictive request that $u_k$ approximate $u$ in the strict topology, it is straightforward that 
 $$\mathcal F(u,\Om)\geq \mathcal F^{L^1}(u,\Om),$$
and strict inequality often occurs. In fact also the domain of $\mathcal F(u,\Om)$ is strictly smaller than that of $\mathcal F^{L^1}(u,\Om)$ (precisely, there exists $u\in BV(\Om;\R^m)$ for which $\mathcal A^{L^1}(u,\Om)$ is finite and $\mathcal A(u,\Om)$ is $+\infty$, see \cite{BCS2}). 

As a consequence of \eqref{19}, for the relaxed area functional $\mathcal A(u,\Om)$, it holds
\begin{align}\label{lbstrict}
\mathcal A(u,\Om)\geq |S_{{\rm strict}}|=|G_u|+|V_{{\rm strict}}|=\int_\Om|\mathcal M(\nabla u)|dx+|V_{{\rm strict}}|.
\end{align}
This provides a natural lower bound for $\mathcal A(u,\Om)$, since $V_{{\rm strict}}$ is uniquely determined by $u$. However, it has been observed \cite{Mucci}
that also in this case the strict inequality can occurs in \eqref{lbstrict}, so the lower bound is not optimal (see also \cite{BCS,BCS2,C}). On the other hand, following the analysis of \cite{BCS,BCS2,C}, in the case that $\Om\subset\R^2$, all the phenomena related to non-subadditivity of the set function $A\mapsto \mathcal A(u,A)$ seemed to disappear, at least for a suitable class of maps of bounded variation $u$, so it has been conjectured that actually the set function $A\mapsto \mathcal A(u,A)$ is the trace of a Borel measure restricted to the class of open sets. This conjecture has been disproved in the case $\Om\subset\R^3$ in \cite{CM}, where the authors show that already for the vortex map $u_V(x)=\frac{x}{|x|}$ some similar phenomena as in dimension $2$  for $\mathcal A^{L^1}$ take place. However it remained an open problem to undestand if in dimension $2$ the conjecture is true.

 In the present paper we show this conjecture, which actually applies also for the more general polyconvex functionals $\mathcal F$:

\begin{theorem}\label{teo_main1}
Let $\Om\subset\R^2$ be an open and bounded set, let $m\geq 1$, and let $u\in BV(\Om;\R^m)$; then the function $A\mapsto \mathcal F(u,A)$, defined for all open sets $A\subseteq\Om$, is the restriction of a Borel measure. 
\end{theorem}

The above result applies to all polyconvex functionals of the form \eqref{functionalF} satisfying \eqref{functionalF2} for a general convex function $g$ that is linear or sublinear, in the sense that there exists a positive constant $C_g$ with 
\begin{align}\label{growth1_intro}
g(\mathcal M(\nabla u))\leq C_g(|\mathcal M(\nabla u)|+1).
\end{align}
At the same time, we assume also some coercivity property of $g$ (see \eqref{growth2} below), that in the case in which $n=m=2$, it is expressed as
\begin{align}\label{growth2_intro}
c_g|\det(\nabla u)|\leq g(\mathcal M(\nabla u))
\end{align}
for some positive constant $c_g$ (and that are weaker than \eqref{growth_g}).
With these two requirements we includes in our analysis the interesting prototype cases of the area functional $g(\mathcal M(\nabla u))=|\mathcal M(\nabla u)|$ and of the total variation of the Jacobian functional, i.e., the functional (in the case $n=m=2$)
\begin{align}\label{TVJ}
TVJ(u,\Om):=\int_\Om|\det(\nabla u)|dx,
\end{align}
defined for $u:\Om\rightarrow \R^2$.

In order to show Theorem \ref{teo_main1} we apply the standard result due to De Giorgi and Letta which characterizes the maps on open sets which are Borel measures  (see Theorem \ref{DGL_teo} below). 
This accounts to check monotonicity, additivity, subadditivity, and inner regularity of the set function $A\mapsto \mathcal F(u,A)$, defined for $A$ open. Although additivity on disjoint set is straightforward, notice that already monotocity is non-trivial, due to the fact that, if $B\subset A$, the restriction of a recovery sequence for $\mathcal F(u,A)$ to $B$ is not necessarily converging strictly to $u$ on $B$. So, accurate modifications of recovery sequence are necessary.

A fundamental step to show subadditivity and inner regularity is Proposition \ref{prop:recoveryrestriction}. Under suitable conditions on $u$ and $B\subset\subset A$, it states that if $u_k$ is a recovery sequence for $\mathcal F(u,A)$ and $u_k\res \partial B$ strictly converges to $u\res \partial B$, then $u_k\res B$ is a recovery sequence for $\mathcal F(u,B)$. To prove Proposition \ref{prop:recoveryrestriction} we assume that $v_k$ is a recovery sequence for $\mathcal F(u,B)$  and we consider a map $w_k$ obtained by glueing $v_k$ and $u_k\res (A\setminus B)$ on a tubular neighborhood of $\partial B$. We show that this can be done by modifying $v_k$ and $u_k\res (A\setminus B)$ a little bit so that their energy does not increase too much; this is possible thanks to the assumption of strict convergence of $u_k$ to $u$ on $\partial B$, since Proposition \ref{prop_convstrict} allows to reparametrize $u_k\res \partial B$ in such a way that it can be glued to $v_k\res \partial B$ by a tricky interpolation argument. This is a crucial point, which is possible only because the set $\partial B$ is $1$-dimensional, and this argument fails in the case $B\subset \R^n$ with $n\geq3$ (this is related with the fact that a the total minimal lifting of $u$ is unique, see \cite{Mucci}, that is not true in dimension greater than $2$). 
To apply the previous interpolation between $v_k$ and $u_k\res (A\setminus B)$ we need that $v_k\res \partial B$ also converges to $u$ strictly on $\partial B$. This is not always true, and requires an ad hoc modification of a recovery sequence $v_k$ for $\mathcal F(u,B)$. 
A key ingredient  in order to modify recovery sequences is the fact that strict convergence on an open set $A\subset\R^2$ is inherited on suitable curves $\Gamma\subset A$. This allows to conclude that $v_k$ converges strictly to $u$ on almost every level set of the distance function $d(\cdot,\partial B)$. With ad hoc transformation in tubular neighborhood of $\partial B$, we can then modifying $v_k$, not changing $F(v_k,B)$ too much, in order that the modified sequence converges strictly to $u$ on $\partial B$ (see Lemma \ref{lem_recoverymodified}).

In view of Theorem  \ref{teo_main1} we expect that, at least for the area and total variation of the Jacobian functional, a suitable integral representation is possible. We provides in Section \ref{sec_representations} some examples of known results. Using these, it is possible to show that for the standard relaxation of the area functional with respect to the $L^1$ convergence, the presence of singular contribution is not only due to the presence of holes (or singularities) in the graph of the considered map. Indeed, even if a map $u:\Om\rightarrow \R^2$ is Cartesian (i.e., its graph $G_u$ has not holes, namely $\partial G_u=0$ as current in $\mathcal D_1(\Om\times \R^2)$), it is possible that the relaxed area $\mathcal A^{L^1}(u,\Om)$ is strictly larger than the $2$-dimensional hausdorff measure of $G_u$ (in other words, a singular contribution due to relaxation pops up). This is our second main result, summarized in Theorem \ref{teo_main2} in Section \ref{sec_representations}.

We emphasize that an integral representation of this kind of functionals as in \cite{DalMaso:80} is not possible if we relax with respect to the $L^1$-topology, due to the lack of sub-additivity of $A\mapsto \mathcal F^{L^1}(u,A)$, unless one requires more restrictive growth conditions on $g$ (see for instance \cite{Fons1,Fons2,soneji}).

The structure of the paper is as follows: In the next Section \ref{sec_notation} we introduce some standard notation and in its Subsection \ref{sec_setting} we recall the setting of the problem. In Section \ref{sec_preliminaryresults} we start with measure theoretic, geometry tools, and preliminary results; further in Section \ref{subsec_properties} we start by describing of to modify Lipschitz maps in order to cut and paste suitable recovery sequences for $\mathcal F(u,\Om)$.
In Section \ref{sec_proof} we finally give the proof of Theorem \ref{teo_main1}, exploiting De Giorgi and Letta Theorem, and thus checking that standard conditions of the set map $A\mapsto \mathcal F(u,A)$ are satisfied. In Section \ref{sec_representations} we exhibit some known result of representation formulas for the area functional (and for the total variation of the Jacobian one); motivated by this, we introduce the double $8$-curve map $u_\varphi$, which is a $0$-homogeneous Cartesian map and we show in Theorem \ref{teo_main2} that $$\mathcal A^{L^1}(u_\varphi,B_r(0))>\int_{B_r(0)}\sqrt{1+|\nabla u_\varphi|^2}dx.$$
The paper ends with an Appendix where we collect a couple of standard results used in the manuscript.

\section{Notation and Setting}\label{sec_notation}

\subsection{Notation}

In what follows we denote by $\mathcal L^n$ the Lebesgue measure  and, for $0\leq d\leq n$, by $\mathcal H^{d}$ the $d$-dimensional Hausdorff measure in $\R^n$. 
Let $A\subseteq\R^n$ be an open set and let $M\geq 1$, we denote by $\mathcal M_b(A;\R^M)$ the space of Radon measures with bounded total variations, and if $\mu\in \mathcal M_b(A;\R^M)$ we denote by $|\mu|(U)$ its total variation on $U\subseteq A$.
\medskip

\textbf{Functions of bounded variation:} We will recall the main properties of functions of bounded variation, and we refer to  \cite{AmFuPa:00} for more detail.
Let $A\subseteq\R^n$ be an open set and let $u\in BV(A;\R^m)$ be a map. We denote by $Du$ the distributional derivative of $u$ which splits as
\begin{align*}
Du=\nabla u+D^cu+D^ju,
\end{align*}  
where $\nabla u$ is the approximate gradient (i.e. the absolutely continuous part of $Du$ with respect to $\mathcal L^n$), $D^cu$ is the Cantor part, and $D^ju$ the jump part of $Du$. The jump set of $u$ is denoted by $S_u\subset A$ and it is a ${(n-1)}$-rectifiable set; 
if $\nu$ is a unit vector normal to $S_u$ at $x\in S_u$, then we denote 
$$u^+(x):=\text{{\rm aplim}}_{y\rightarrow x,\;(y-x)\cdot \nu>0}\;u(x),\qquad u^-(x):=\text{{\rm aplim}}_{y\rightarrow x,\;(y-x)\cdot \nu<0}\;u(x)$$
and so it turns out that 
\begin{align*}
D^ju=(u^+-u^-)\otimes \nu \cdot \mathcal H^{n-1}\res S_u.
\end{align*}
We denote by $|Du|(A)$ the total variation of $u$ in $A$, that coincides with 
\begin{align}
|Du|(A)=\sup\{\sum_{i=1}^m\int_A u_i \cdot \div\varphi_i dx: \varphi\in C^1_c(A;\R^{m\times n}),\;\|\varphi\|_{L^\infty}\leq 1\}
\end{align}
where $\varphi_i$ denotes the $i$-th row of $\varphi$.

In the one dimensional case $n=1$ the jump set $S_u$ reduces to an at most countable (possibly empty) subset  of $A$. If $t\in A$ we denote 
$$u(t^+):=\lim_{x\rightarrow t^+}u(x)\qquad u(t^-):=\lim_{x\rightarrow t^-}u(x),$$
so that $D^ju=\sum_{t\in S_u}(u(t)^+-u(t)^-)\delta_t=\sum_{t\in S_u}(u(t^+)-u(t^-))\delta_t$. In the one dimensional case there exists always a good representative of $u$ that is right-continuous, and its only discontinuity points are those in the jump set.

\begin{definition}
We say that a sequence  $u_k\subset BV(A;\R^m)$ converges to $u\in BV(A;\R^m)$ strictly in  $BV(A;\R^m)$ if 
\begin{align*}
u_k\rightarrow u \text{ in }L^1(A;\R^m),\qquad \quad |Du_k|(A)\rightarrow |Du|(A),
\end{align*}
when $k\rightarrow\infty$. 
\end{definition}
The topology induced by the strict convergence is metrizable and we denote by $d_s$ the distance associated with it: Specifically, for $u,v\in BV(A;\R^m)$ we set
\begin{align}\label{strictdistance}
d_s(u,v):=\|u-v\|_{L^1}+\big||Du|(A)-|Dv|(A)\big|.
\end{align}
With this notation $u_k\rightarrow u$ strictly in $BV(A;\R^m)$ if and only if $d_s(u_k,u)\rightarrow 0$.

We recall the following approximation result:
\begin{theorem}\label{approx_teoBV}
Let $A\subset\R^n$ be a bounded open set, and let $u\in BV(\Om;\R^m)$. Then there exists a sequence $(v_k)\subseteq C^\infty(A;\R^m)$ such that $v_k\rightarrow u$ strictly in $BV(A;\R^m)$.
\end{theorem}
Inspecting the proof of the Theorem above (see, e.g.,  \cite{AmFuPa:00}), the following remark is in order:
\begin{remark}\label{rem_approximation}
The previous Theorem is obtained by a local argument of mollification and then using a unity partition. In particular, if $u$ is Lipschitz continuous in $A$, then
$$v_k\rightarrow u\qquad \qquad \text{weakly* in }W^{1,\infty}(A;\R^{m})\text{ and strongly in }W^{1,p}(A;\R^m),$$  
for all $p<\infty$, 
and the functions $v_k$ are Lipschitz continuous with Lipschitz constant less than or equal to the one of $u$. 
\end{remark}
\medskip

\textbf{Currents:} 
For an open set $A\subset \R^n$ we denote by $\mathcal D^k(A)$ the 
space of (compactly 
supported in $A$) smooth $k$-forms and by $\mathcal D_k(A)$ the space of $k$-dimensional currents, where $0\leq k\leq n$. Given $T\in \mathcal D_k(\R^n)$ we denote by $\vert T\vert_{\R^n}$  the mass of $T$, and by $\vert T\vert_A$ 
its mass in an open set $A\subset \R^n$.
%
Given $T\in \mathcal D_k(A)$ with  $k\geq 1$, its 
 boundary $\partial T\in \mathcal D_{k-1}(A)$ 
is defined by 
$$\partial T(\omega):=T(d \omega)\qquad  \forall\omega\in \mathcal D^{k-1}(A),$$
where $d\omega$ denotes the external differential of $\omega$. 
In the case $k=0$ by convention it is $\partial T=0$.
Whenever $F:A\rightarrow B$ is a Lipschitz 
map between open sets, and $T\in \mathcal D_k(A)$, the symbol $F_\sharp T\in \mathcal D_k(B)$ denotes the push-forward of $T$ by $F$.
%

We say that a current $T\in \mathcal D_k(A)$ is rectifiable if there exist a $\mathcal H^k$-rectifiable set\footnote{$S$ is said $\mathcal H^k$-rectifiable if there are (at most) countably many Lipschitz maps $\phi_h:\R^k\rightarrow \R^n$ such that 
	$$S\subseteq N\cup\bigcup_{h=0}^{+\infty} \phi_h(\R^k),\qquad \mathcal H^k(N)=0.$$
	} $S$, a simple unit $k$-vector $\tau(x)$ 
for $\mathcal H^k$-a.e.
$x\in S$, and a measurable function $\theta:S\rightarrow \R$ with
\begin{equation*}\label{eq:S_rectif}
	T(\omega)=\int_S\theta(x)\langle\omega(x),\tau(x)\rangle~ 
	d\mathcal H^{k}(x),\qquad \omega\in \mathcal D^k(A).
\end{equation*}
A rectifiable current  $T\in \mathcal D_k(A)$ is said integral if $\theta$ takes integer values, $\tau$ is tangent to $S$, and   $|T|_A<+\infty$, $|\partial T|_A<+\infty$.
In the special case in which $S=E$ is a finite subset of $\R^n$,  we denote by $\jump{E}$ the standard integration over $E$ defined as the rectifiable $n$-current with $\theta=1$ and $\tau=e_1\wedge\dots\wedge e_n$ is the standard orientation of $\R^n$. Precisely
$$	\jump{E}(\omega)=\int_E\langle\omega(x),e_1\wedge\dots\wedge e_n\rangle~ 
dx,\qquad \omega\in \mathcal D^n(\R^n). $$  
If $E$ is a finite perimeter set with finite Lebesgue measure, then $\jump{E}$ turns out to be an integral current.
\medskip

\textbf{Graphs and Cartesian maps:}
Let $m\geq 2$ be a fixed integer; multi-indeces $\alpha\subseteq\{1,\dots,n\}$ and $\beta\subseteq\{1,\dots,m\}$ are two ordered sets, possibly empty. We denote by $|\cdot|$ the cardinality; by $\overline\alpha$ we denote the complementary of $\alpha$, i.e. $\overline \alpha:=\{1,\dots,n\}\setminus \alpha$, and similarly $\overline\beta:=\{1,\dots,m\}\setminus \beta$. Given a $m\times n$ matrix $A=(a_{ij})$, $i\in \{1,\dots,m\}$, $j\in \{1,\dots,n\}$, and given $\alpha,\beta$ multi-indeces as above such that $|\alpha|+|\beta|=n$, we denote by $$M_{\overline \alpha}^\beta(A),$$
the determinant of the submatrix of $A$ whose columns are indexed in $\overline \alpha$ and lines in $\beta$,  multiplied by $\theta(\alpha)$, the sign of the permutation $(\alpha,\overline\alpha)\in S(n)$ (with the convention that $M_\varnothing^\varnothing(A)=1$). In the specific case of our interest, if $n=2$ and  $A=\nabla u$, with $u:\R^2\rightarrow \R^m$ a sufficiently smooth map, it holds
\begin{align*}
M_\varnothing^\varnothing(A)=1\qquad \qquad 	M_{j}^{i}(\nabla u)=(-1)^j\frac{\partial u_i}{\partial x_j}\qquad \qquad M_{12}^{i_1i_2}(\nabla u)=\frac{\partial u_{i_1}}{\partial x_1}\frac{\partial u_{i_2}}{\partial x_2}-\frac{\partial u_{i_2}}{\partial x_1}\frac{\partial u_{i_1}}{\partial x_2}.
\end{align*}
We denote by $\{e_1,\dots, e_n\}$ the canonical basis of $1$-vectors of $\R^n$, and by $\{\varepsilon_1,\dots,\varepsilon_m\}$ that of the target space $\R^m$. The dual basis of $1$-covectors are denoted by $\{dx_1,\dots,dx_n\}$ and $\{dy_1,\dots,dy_m\}$, respectively. If $\alpha\subseteq \{1,\dots,n\}$ and $\beta\subseteq\{1,\dots,m\}$ are ordered sets as above, we denote $e_\alpha$ and $\varepsilon_\beta$ the $k$-vector and $h$-vector defined as 
\begin{align}
e_\alpha:=e_{\alpha_1}\wedge\dots\wedge e_{\alpha_k}\qquad \text{ if }\alpha=\{\alpha_1,\dots,\alpha_k\},\\
\varepsilon_\beta:=\varepsilon_{\beta_1}\wedge\dots\wedge \varepsilon_{\beta_h}\qquad \text{ if }\beta=\{\beta_1,\dots,\beta_h\},
\end{align}
where $k=|\alpha|$, $h=|\beta|$, so in the case $n=2$ it holds
\begin{align}
e_\varnothing=1,\qquad e_\alpha=e_j\;\;\text{ if }\alpha=\{j\},\qquad e_{12}=e_1\wedge e_2.
\end{align}
Next we introduce the $n$-vector associated to a $C^1$ map  $u:\R^n\rightarrow \R^m$
$$\mathcal M(\nabla u):=\sum_{|\alpha|+|\beta|=n}M_{\overline \alpha}^\beta(\nabla u)e_\alpha\wedge \varepsilon_\beta,$$
where the sum takes place over all multi-indeces $\alpha\subseteq\{1,\dots,n\}$ and $\beta\subseteq\{1,\dots,m\}$ with $|\alpha|+|\beta|=n$.

Given a map  $u\in C^1(A;\R^m)$ we introduce its graph $G_u\subseteq A\times \R^m$ as
$$G_u=\{(x,y)\in A\times \R^m:y=u(x)\}$$ and we use the map $\text{{\rm Id}}\times u:A\rightarrow A\times \R^m$, $(\text{{\rm Id}}\times u)(x):=(x,u(x))$, to parametrize it.
$G_u$ is
 identified in a natural way with an integral current given by integration over it. More precisely, denoting this current by $\jump{G_u}$, its standard orientation  is given by $\mathcal M(\nabla u)/|\mathcal M(\nabla u)|$, the multiplicity $\theta$ is always $1$, and so for all $n$-form $\omega\in \mathcal D^n(A\times \R^m)$ it holds
$$\jump{G_u}(\omega)=(\text{{\rm Id}}\times u)_\sharp\jump{A}=\int_A\langle\omega(x,u(x)),\mathcal M(\nabla u(x))\rangle\, dx.$$
It is seen that  $\jump{G_u}$ hass mass that coincides with the  $\mathcal H^n$-measure of $G_u$, and is given by 
$$	|\jump{G_u}|=\area(u,A)=\int_A|\mathcal M(\nabla u)|~dx.$$
It turns out, thanks to the regularity of $u$, that $\jump{G_u}$ is boundaryless.

We now want to extend the definitions above for maps $u\in BV(A,\R^m)$. To this aim we denote by $R_u\subseteq A$ the set of regular points of $u$, namely the  points $x$ that are Lebesgue points for $u$ and $\nabla u$, moreover $u(x)$ 
coincides with its Lebesgue value and $u$ 
is approximately differentiable at $x$. We denote 
\begin{align*}
	G_u^R:=\{(x,y)\in R_u\times \R^2:y=u(x)\}.
\end{align*}
Also $G_u^R$ is $\mathcal H^n$-rectifiable
and we define $$\mathcal G_u:=\jump{G_u^R}=(\text{{\rm Id}}\times u)_\sharp\jump{R_u}.$$ It holds that  
$$	|\mathcal G_u|=\int_A|\mathcal M(\nabla u)|~dx,$$
where $\nabla u$ is the approximate gradient of $u$. In general $\mathcal G_u$ has non-trivial boundary. In the special case that $\partial\mathcal G_u=0$ in $\mathcal D_{n-1}(A\times \R^m)$ we say that $u$ is a Cartesian map. 

\subsection{Relaxation and approximation}
In this section we are concerned with the relaxation of the functional
\begin{align}\label{functional_F}
	F(u,\Om):=\int_\Om g(\mathcal M(\nabla u))dx,
\end{align}
where $g$ is a convex function satisfying \eqref{growth1_intro}.
Standard relaxation in the space $BV(\Om;\R^m)$ with respect to the strict convergence is given by \eqref{rel_strict}, where the functions $u_k$ are obviously taken in $C^1(\Om;\R^m)\cap BV(\Om;\R^m)$, since we approximate $u$ in the strict topology.

We now observe that the constraint in \eqref{rel_strict} of taking approximating functions $u_k\in C^1(\Om;\R^m)$ can be weakened. To this purpose, for simplicity we restrict to the case of interest of this paper, namely $\Om\subset\R^2$, even if the same discussion can be done for the case $n\geq3$.
We introduce the alternative relaxation, that is, for all $u\in BV(\Om;\R^m)$,
\begin{align}
	\mathcal F^*(u,\Om):=\inf\{\liminf_{k\rightarrow\infty}F(u_k,\Om):(u_k)\subset \text{{\rm Lip}}_{\text{{\rm loc}}}(\Om;\R^m),\;u_k\rightarrow u\text{ strictly in }BV(\Om;\R^m)\}.
\end{align}
Let $u\in \text{{\rm Lip}}_{\text{{\rm loc}}}(\Om;\R^m)\cap BV(\Om;\R^m)$: By Remark \ref{rem_approximation}, there exists a sequence $(v_k)\subset C^1(\Om;\R^m)\cap BV(\Om;\R^m)$ such that $v_k\rightarrow u$ strictly in $BV(\Om;\R^m)$ and
\begin{align*}
&\nabla v_k\rightarrow \nabla u\qquad \qquad \text{ strongly in } L^1(\Om;\R^{m\times 2}),\\
&M_{12}^{ij}(\nabla v_k)\rightarrow M_{12}^{ij}(\nabla u)\qquad \text{ strongly in }L^1(\Om),
\end{align*}
 for all $i,j\in \{1,\dots,m\}$. Up to a subsequence these convergences take place also poitwise a.e., and by \eqref{growth1_intro} we can apply Lebesgue dominated convergence theorem to conclude
 \begin{align}
F(v_k,\Om)\rightarrow F(u,\Om).
 \end{align}
 As a consequence, if $(u_j)\subset \text{{\rm Lip}}_{\text{{\rm loc}}}(\Om;\R^m)\cap  BV(\Om;\R^m)$ is a recovery sequence for $\mathcal F^*(u,\Om)$, by a diagonal argument we can find a sequence $(v_k)\subset C^1(\Om;\R^m)\cap BV(\Om;\R^m)$ such that 
$
  F(v_k,\Om)\rightarrow \mathcal F(u,\Om).
 $ We conclude that $\mathcal F^*(u,\Om)\geq \mathcal F(u,\Om)$.
 
 Viceversa,  it is immediate that $\mathcal F^*(u,\Om)\leq \mathcal F(u,\Om)$ (since every $C^1$ function is also locally Lipschitz).  Whence $\mathcal F^*=\mathcal F$.
 Thanks to this observation, we can often consider locally Lipschitz recovery sequence instead of maps of class $C^1$. 
 
%
%
%
%
%

\subsection{Setting of the problem}\label{sec_setting}
In what follows $\Om\subset\R^2$ will be our reference domain, an open bounded set.
Let $N:=1+2m+m(m+1)/2$ and let $g:\R^{N}\rightarrow [0,+\infty)$ be convex; our functional is given by \eqref{functional_F}
whenever $u\in C^1(\Om;\R^m)\cap BV(\Om;\R^m)$. To extend it on $BV(\Om;\R^m)$, we proceed by relaxation and consider the functional $\mathcal F(u,\Om)$ given in \eqref{rel_strict}. This turns out to be lower-semicontinuous with respect to the strict convergence in $BV(\Om;\R^m)$.
%
To our purposes, we will assume that there is a constant $C_g>0$  such that for all $A\in \R^N$,
\begin{align}\label{growth1}
|g(A)|\leq C_g(|A|+1).
\end{align}
Furthermore, we assume not degeneracy of the functional through the folllowing condition
\begin{align}\label{growth2}
|g(A)|\geq c_g\sum_{\substack{i,j=1\\i\neq j}}^{m}|M_{12}^{ij}(A)|,
\end{align}
for a general positive constant $c_g$. In the case that $m=2$ the above condition is equivalent to \eqref{growth2_intro}.
As a consequence of the growth condition \eqref{growth1} and of the convexity of $g$, the subdifferential $\partial g$ satisfies
\begin{align}\label{growthbis}
\|\partial g\|_{L^\infty}\leq C_g.
\end{align}

\section{Tools and preliminary results}\label{sec_preliminaryresults}

\subsection{Properties of measures}
In order to prove our main result Theorem \ref{teo_main1} we will employ the classical theorem named after De Giorgi and Letta, which we collect here in a form specialized for our setting (see \cite[Theorem 1.53]{AmFuPa:00} for the general formulation and its proof).
We denote by $\mathcal U(\Om)$ the family of open subsets of $\Om$.
\begin{theorem}[De Giorgi-Letta]\label{DGL_teo}
	Let $\Om\subset\R^2$ be an open set and assume that $\mu:\mathcal U(\Om)\rightarrow [0,+\infty]$ is a function so that $\mu(\varnothing)=0$. If
	\begin{itemize}
\item[(i)] $\mu$ is non-decreasing, i.e., $\mu(B)\leq \mu(A)$ for all $A,B\in \mathcal U(\Om)$, $B\subseteq A$;
\item[(ii)] $\mu$ is additive, i.e., $\mu(A\cup B)=\mu(A)+\mu(B)$ for all $A,B\in \mathcal U(\Om)$, $A\cap B=\varnothing$;
\item[(iii)] $\mu$ is sub-additive, i.e., $\mu(A)\leq \mu(B_1)+\mu(B_2)$ for all $A,B_1,B_2\in \mathcal U(\Om)$, $A\subseteq  B_1\cup B_2$;
\item[(iv)] $\mu$ is inner regular, i.e., for all $A\in \mathcal U(\Om)$ it holds
$$\mu(A)=\sup\{\mu(B):B\in \mathcal U(\Om),\;B\subset\subset A\};$$
	\end{itemize}
Then $\mu$ is the restriction to $\mathcal U(\Om)$ of a Borel measure $\overline \mu:\mathcal B(\Om)\rightarrow [0,+\infty]$.
\end{theorem}

We will often use the following result due to Reshetnyak (see \cite{AmFuPa:00}[Theorem 2.39]):
\begin{theorem}\label{teo_32}
Let $M\geq1$ and let $\mu,\mu_k$ be Radon measures in $A\subseteq \R^n$ taking values in  $\R^M$. Suppose that $\mu_k\rightharpoonup \mu$ weakly star as measures and that $|\mu_k|(A)\rightarrow |\mu|(A)$. Then 
\begin{align*}
\int_Af\left(x,\frac{\mu_k}{|\mu_k|}(x)\right)d|\mu_k|(x)\rightarrow \int_Af\left(x,\frac{\mu}{|\mu|}(x)\right)d|\mu|(x)
\end{align*}
as $k\rightarrow \infty$ for all continuous and bounded functions $f:A\times S^{M-1}\rightarrow \R$.
\end{theorem}
We will also need the following property valid for strictly converging Borel measures $\mu_k,\mu$.
\begin{lemma}
Suppose that $\mu_k\rightharpoonup \mu$ weakly star as measures and $|\mu_k|(A)\rightarrow |\mu|(A)$, and let $B\subset A$ be open. Then if $|\mu|( A\cap \partial B)=|\mu_k|( A\cap \partial B)=0$ for all $k$, it holds
$$|\mu_k|(B)\rightarrow |\mu|(B).$$
\end{lemma}
\begin{proof}
By lower semicontinuity of the total variation on open sets and thanks to the hypothesis $\mu( A\cap \partial B)=0 $ we have
\begin{align*}
	|\mu|(A)&=|\mu|(B)+|\mu|(A\setminus \overline B)\leq \liminf_{k\rightarrow \infty}|\mu_k|(B)+\liminf_{k\rightarrow \infty}|\mu_k|(A\setminus \overline B)\\
&\leq \liminf_{k\rightarrow \infty}|\mu_k|(A)=\lim_{k\rightarrow \infty}|\mu_k|(A)=|\mu|(A),
\end{align*}
so all the inequalities are equalities and in particular $|\mu|(B)=\liminf_{k\rightarrow \infty}|\mu_k|(B)$. Since the same holds for every subsequence of $\mu_k$, we easily infer that the liminf is indeed a limit.
\end{proof}

We also collect the following result which can be found in \cite[Proposition 1, Section 1.3.4]{GMS3}.

\begin{prop}\label{prop34}
	Let $A$ be open and bounded and let $h$ be a positive integer.
Let $V_k,V\in L^1(A;\R^h)$ be such that $V_k\rightharpoonup V$ weakly star in $L^1(A;\R^h)$ and moreover
\begin{align*}
\int_A\sqrt{1+|V_k|^2}dx\rightarrow \int_A\sqrt{1+|V|^2}dx
\end{align*}
as $k\rightarrow +\infty$. Then $V_k\rightarrow V$ strongly in $L^1(A;\R^h). $
\end{prop}

\subsection{Lipschitz and BV curves}
Given a Lipschitz map $\varphi:[a,b]\rightarrow \R^m$, we denote by $L_\varphi:=\int_a^b|\dot\gamma|d\tau$ its total variation and we introduce the quantity 
\begin{align}\label{s_varphi}
s_\varphi(t)=\frac{1}{L_\varphi+(b-a)}\int_a^t(|\dot \varphi|+1)d\tau, \qquad \qquad \forall t\in [a,b].
\end{align}
This is a strictly increasing and continuous function, so we let $t_\varphi:[0,1]\rightarrow[a,b]$ be its inverse $t_\varphi=s_\varphi^{-1}$, which satisfies
\begin{align}\label{t_varphi}
	\dot t_\varphi(s)=\frac{L_\varphi+(b-a)}{|\dot \varphi(t_\varphi(s))|+1}\qquad \qquad \forall s\in [0,1].
\end{align} 
In particular $	\dot t_\varphi(s)\leq L_\varphi+(b-a)$ for all $s\in [0,1]$. A similar definition applies to a function
 $\gamma\in BV([a,b];\R^m)$, for which  we denote  $L_\gamma:=|\dot \gamma|([a,b])$ and 
\begin{align}\label{s_gamma}
s_\gamma(t)=\frac{1}{L_\gamma+(b-a)}\big(|\dot\gamma|([a,t))+(t-a)\big), \qquad \qquad \forall t\in [a,b],
\end{align}
which is strictly increasing with jumps set $S_\gamma$, the jump set of $\gamma$; moreover 
$$s_\gamma(t_1)-s_\gamma(t_2)\geq \frac{t_1-t_2}{L_\gamma+(b-a)},\qquad \qquad 0\leq t_2\leq t_1\leq1,$$
and so it follows that if  $ t_\gamma:=s_\gamma^{-1}:[0,1]\rightarrow [a,b]$ is the inverse of $s_\gamma$ that is constant on $[s_\gamma(t^-),s_\gamma(t^+)]$, for all $t\in S_\gamma$, we have
\begin{align*}
t_\gamma(s_1)-t_\gamma(s_2)=|t_\gamma(s_1)-t_\gamma(s_2)|\leq (s_1-s_2)(L_\gamma+(b-a)),\qquad \qquad 0\leq s_2\leq s_1\leq1.
\end{align*}
Hence $t_\gamma$ is Lipschitz continuous with Lipschitz constant $L_\gamma+(b-a)$.
\begin{definition}
Given $\gamma\in BV([a,b];\R^m)$ we define $\overline \gamma:[0,1]\rightarrow \R^m$ as  
\begin{align}\label{gencurve}
	\overline \gamma(s)=\begin{cases}
		\frac{\gamma(t^+)(s-s_\gamma(t^-))+\gamma(t^-)(s_\gamma(t^+)-s)}{s_\gamma(t^+)-s_\gamma(t^-)}&\text{if }s\in [s_\gamma(t^-),s_\gamma(t^+)],\\
		\gamma(t_\gamma(s))&\text{otherwise.}
	\end{cases}
\end{align}
\end{definition}
Obviously this definition applies also when $\gamma=\varphi$ is Lipschitz continuous, and in this case 
it simply holds $\overline \varphi(s)=\varphi(t_\varphi(s))$ that is Lipschitz continuous and satisfies
\begin{align}
\left|\frac{d}{ds}\overline \varphi(s)\right|=\left|\dot \varphi(t_\varphi(s))\dot t_\varphi(s)\right|\leq L_\varphi+(b-a),\qquad \qquad \text{for a.e. }s\in [0,1].
\end{align}
The same is true for $\overline \gamma$ when $\gamma\in  BV([a,b];\R^m)$; we will obtain this as a consequence of the following result.

\begin{prop}\label{prop_convstrict}
Let $\gamma\in BV([a,b];\R^m)$ and let $(\varphi_k)\subset \textrm{Lip}([a,b];\R^m)$ be a sequence of maps converging strictly to $\gamma$ as $k\rightarrow \infty$. The functions $\overline \varphi_k:=\varphi_k\circ t_{\varphi_k}:[0,1]\rightarrow \R^m$ are Lipschitz continuous with uniformly bounded Lipschitz constants and \begin{align}\label{conv_curves}
&\overline\varphi_k\rightarrow \overline \gamma\qquad \text{strictly in $BV([0,1];\R^m)$ and weakly star in }W^{1,\infty}([0,1];\R^m),\nonumber\\
&s_{\varphi_k}\rightarrow s_\gamma\qquad \text{strictly in }BV([a,b]),\\
&t_{\varphi_k}\rightarrow t_\gamma\qquad \text{weakly star in }W^{1,\infty}([0,1]).\nonumber
\end{align}  
Moreover there exists a function $a_\gamma:\R^+\rightarrow \R^+$ depending only on $\gamma$ and such that $a_\gamma(t)\rightarrow 0$ when $t\rightarrow0^+$, and 
$$\|s_{\varphi}-s_\gamma\|_{L^1}+\|\overline\varphi-\overline \gamma\|_{L^\infty}\leq a_\gamma(d_s(\varphi,\gamma)),$$
for all $\varphi\in \textrm{Lip}([a,b];\R^m)$. 
\end{prop}

 We remark that Proposition \ref{prop_convstrict} can be obtained by inspecting the arguments leading to \cite[Lemma 2.10]{C} and \cite[Lemma 2.7]{BCS2}. For the reader convenience and for the sake of completeness we give the proof. 
\begin{proof}
	Let us denote $L_\gamma:=|\dot\gamma|([a,b])$, and $L_k:=|\dot \varphi_k|([a,b])$ the total variations of $\gamma$ and $\varphi_k$ respectively. To shortcut the notation we denote $s_{\varphi_k}:[a,b]\rightarrow [0,1]$ in \eqref{s_varphi} by $s_k=s_{\varphi_k}$ and its inverse $t_{\varphi_k}:[0,1]\rightarrow [a,b]$ in \eqref{t_varphi} as $t_k=t_{\varphi_k}$. Moreover we recall the definition of $s_\gamma\in BV([a,b])$ given in \eqref{s_gamma}.
	
	\textit{Step 1: Convergence of $s_{\varphi_k}$ and $t_{\varphi_k}$.} Thanks to the strict convergence of $\varphi_k$ to $\gamma$, it is easy to see that $s_k\rightarrow  s_\gamma$ pointwise a.e. and strictly in $BV([a,b])$. In particular, if $\gamma$ is continuous at $t\in[a,b]$, then $s_k(t)\rightarrow s_\gamma(t)$. 
	Moreover, $s_\gamma$ is strictly increasing, and its jump set coincides with the jump set $S_\gamma$ of $\gamma$.
	
	As for $t_k$, due to the fact that its Lipschitz constant is less than or equal to $L_k+(b-a)$, and since $L_k\rightarrow L_\gamma$, we conclude that there is a Lipschitz function $\tau:[0,1]\rightarrow [a,b]$ such that, up to a subsequence, 
	$$t_k\rightharpoonup \tau\qquad \text{ weakly star in }W^{1,\infty}([0,1]),$$
	and hence also pointwise on $[0,1]$. We claim that $\tau=t_\gamma=s_\gamma^{-1}$, and so, by uniqueness of the limit, we will also infer that the whole sequence $t_k$ converges to $t_\gamma$.   
	
	Notice that $\tau$ is a non-decreasing and continuous mapping $[0,1]$ onto $[a,b]$; let then $\sigma\in [0,1]$ be so that $\tau(\sigma)\notin S_\gamma$. Therefore, for any $\varepsilon>0$ we can find $0<\delta\leq \varepsilon$ so that $I_\delta=(\tau(\sigma)-\delta,\tau(\sigma)+\delta)$ enjoies $|\dot\gamma|(I_\delta)<\varepsilon$, and in addition $\tau(\sigma)-\delta\notin S_\gamma$ and $\tau(\sigma)+\delta\notin S_\gamma$. The last condition implies that $|\dot\varphi_k|(I_\delta)\rightarrow|\dot\gamma|(I_\delta)$, and so
	\begin{align*}
\lim_{k\rightarrow \infty}|s_k(t_k(\sigma))-s_k(\tau(\sigma))|&=\lim_{k\rightarrow \infty}\frac{1}{L_k+(b-a)}\left|\int_{\tau(\sigma)}^{t_k(\sigma)}|\dot\varphi_k|+1dr\right|\\&\leq\frac{1}{L_\gamma+(b-a)} \lim_{k\rightarrow \infty}\int_{I_\delta}|\dot\varphi_k|+1dr\leq\frac{3\varepsilon}{L_k+(b-a)}.
	\end{align*}
	By arbitrariness of $\varepsilon$ we conclude that \begin{align}\label{limitsk}
s_k(\tau(\sigma))\rightarrow s_k(t_k(\sigma))=\sigma\qquad \qquad \text{ as }k\rightarrow\infty.
	\end{align}
On the other hand $s_k(\tau(\sigma))\rightarrow s_\gamma(\tau(\sigma))$, so we conclude $s_\gamma(\tau(\sigma))=\sigma$ for all $\sigma$ with $\tau(\sigma)\notin S_\gamma$. This implies that $\tau(\sigma)=t_\gamma(\sigma)$ for any $\sigma$ such that $\tau(\sigma)\notin S_\gamma$, but now, since $\tau$ is continuous non-decresing and so is $t_\gamma$ (which in addition is constant on the connected components of $t_\gamma^{-1}(S_\gamma)$), necessarily $\tau(\sigma)=t_\gamma(\sigma)$ for all $\sigma\in [0,1]$.

		\textit{Step 2: Convergence of ${\overline \varphi_k}$.}
	Recalling  that 
	$$|\frac{d}{ds}\overline \varphi_k(s)|\leq {L_k+b-a}\qquad\qquad \text{ for a.e. }s\in [0,1],$$
	and since $L_k\rightarrow L_\gamma$ as $k\rightarrow+\infty$, $\overline \varphi_k$ are uniformly bounded in $W^{1,\infty}([0,1];\R^m)$, and so, up to a subsequence, they converge weakly star to some limit $\zeta\in W^{1,\infty}([0,1];\R^m)$ with 
		\begin{align}\label{lunghezzasalto}
		|\frac{d}{ds}\zeta(s)|\leq {L_\gamma+b-a}\qquad\qquad \text{ for a.e. }s\in [0,1].
	\end{align} We have to prove  that this limit is $\overline \gamma$, indipendently from the subsequence; as a consequence it will follow that the full sequence $\overline \varphi_k$ converges to  $\overline \gamma$. 

To this purpose we fix $$\sigma\in [0,1]\setminus\left(\cup_{t\in S_\gamma}[s_\gamma(t^-),s_\gamma(t^+)]\right);$$
this is equivalent to require that $t_\gamma(\sigma)\notin S_\gamma$. Thus we write
\begin{align*}
|\overline \varphi_k(\sigma)-\overline\gamma(\sigma)|&=| \varphi_k(t_k(\sigma))-\gamma(t_\gamma(\sigma))|\leq | \varphi_k(t_k(\sigma))-\varphi_k(t_\gamma(\sigma))|+| \varphi_k(t_\gamma(\sigma))-\gamma(t_\gamma(\sigma))|\\
&\leq \left|\int_{t_k(\sigma)}^{t_\gamma(\sigma)}|\dot \varphi_k|+1 \;dr\right|+| \varphi_k(t_\gamma(\sigma))-\gamma(t_\gamma(\sigma))|\\
&=(L_{k}+(b-a))(s_k(t_k(\sigma))-s_k(t_\gamma(\sigma)))+| \varphi_k(t_\gamma(\sigma))-\gamma(t_\gamma(\sigma))|
\end{align*}
and thanks to \eqref{limitsk} and the fact that $\varphi_k\rightarrow \gamma$ pointwise a.e. on $[a,b]\setminus S_\gamma$, we conclude that 
$$\varphi_k(\sigma)\rightarrow \overline \gamma(\sigma)\qquad \text{ for a.e. }\sigma \in [0,1]\setminus\left(\cup_{t\in S_\gamma}[s_\gamma(t^-),s_\gamma(t^+)]\right).$$
%
%
%
%
%
Therefore we conclude $\zeta=\overline \gamma$ a.e. on  $[0,1]\setminus\left(\cup_{t\in S_\gamma}[s_\gamma(t^-),s_\gamma(t^+)]\right)$.
We want to show that $\zeta(s)$ coincides with the first line in \eqref{gencurve} when $s\in [s_\gamma(t^-),s_\gamma(t^+)]$, for some $t\in S_\gamma$.

	If $t\in S_\gamma$, there are sequences $t^-_j\rightarrow t^-$ and $t^+_j\rightarrow t^+$ as $j\rightarrow \infty$, such that $t^\pm_j$ are continuity points of $\gamma$ (and of $s_\gamma$). In particular 
	$\gamma(t^\pm_j)=\overline \gamma(s_\gamma(t^\pm_j))\rightarrow \overline \gamma(s_\gamma(t)^\pm)$ as $j\rightarrow \infty$,  so 
	$$\overline \gamma(s_\gamma(t)^\pm)=\gamma(t^\pm).$$ Moreover, since $s_\gamma(t)^+= s_\gamma(t)^-+\frac{1}{L_\gamma+b-a}|\dot\gamma|(\{t\})$ we deduce that $$s_\gamma(t)^+- s_\gamma(t)^-=\frac{1}{L_\gamma+b-a}|\gamma(t^+)-\gamma(t^-)|=\frac{1}{L_\gamma+b-a}|\overline \gamma( s_\gamma(t)^+)-\overline \gamma( s_\gamma(t)^-)|.$$
	We conclude that the curve $\overline \gamma\res[s(t^-),s(t^+)]$ is a curve connecting $\overline \gamma( s_\gamma(t)^-)$ to $\overline \gamma( s_\gamma(t)^+)$ on an interval of length $\frac{1}{L_\gamma+b-a}|\overline \gamma(s_\gamma(t)^+)-\overline \gamma( s_\gamma(t)^-)|$; by  \eqref{lunghezzasalto} this curve must necessarily be the constant speed parametrization of the segment with endpoints $\overline \gamma(s_\gamma(t)^-)$ and $\overline \gamma(s_\gamma(t)^+)$, namely $\zeta(s)$ coincides with the interpolation in \eqref{gencurve}. We conclude then also the first thesis in \eqref{conv_curves}.

\textit{Step 3:} To prove the last statement, we set 
$$a_\gamma(t):=\sup\{\|s_\varphi-s_\gamma\|_{L^1}+\|\overline \varphi-\overline \gamma\|_{L^\infty}:\varphi\in \textrm{Lip}([a,b];\R^m),\;d_s(\varphi,\gamma)\leq t\}.$$
Assume by contradiction that there exists a sequence of positive numbers $t_k\searrow0$ such that $\lim_{k\rightarrow \infty}a_\gamma(t_k)>0.$ Then, by definition of $a_\gamma$ we can find functions $\psi_k\in \textrm{Lip}([a,b];\R^m)$ such that $d_s(\psi_k,\gamma)\leq t_k$ and $$ \lim_{k\rightarrow \infty}(\|s_{\psi_k}-s_\gamma\|_{L^1}+\|\overline {\varphi}_k-\overline \gamma\|_{L^\infty})>0.$$
This is a clear contradiction with \eqref{conv_curves}, hence the thesis follows. 
\end{proof}
\begin{cor}\label{cor_varbar}
Let $\gamma\in BV([a,b];\R^m)$, then $\overline \gamma$ is Lipschitz continuous with Lipschitz constant $L_\gamma+(b-a)$.
\end{cor}
\begin{proof}
It is sufficient to approximate $\gamma$ in the strict topology of $BV([a,b];\R^m)$ by Lipschitz maps, and the thesis follows from Proposition \ref{prop_convstrict}.
\end{proof}
%

\textbf{Interpolation between Lipschitz curves:} Let $h>0$ be fixed and let $[a,b]$, $a<b$, be an interval. 
For Lipschitz maps $\varphi,\psi:[a,b]\rightarrow \R^m$ we introduce the following interpolations: $\Phi_{\varphi,\psi}:[a,b]\times [0,h]\rightarrow \R^m$  given by
\begin{align}\label{Phi_interp}
\Phi_{\varphi,\psi}(t,r):=\varphi\Big(t_\varphi\big(s_\varphi(t)\frac{r}{h}+s_\psi(t)\frac{h-r}{h} \big)\Big),
\end{align}
that satisfies $\Phi_{\varphi,\psi}(t,h)=\varphi(t)$ and $\Phi_{\varphi,\psi}(t,0)=\varphi(t_\varphi\circ s_\psi(t))$,  
and the mapping $\Psi_{\varphi,\psi}:[a,b]\times [0,h]\rightarrow \R^m$ defined by
\begin{align}\label{Psi_interp}
	\Psi_{\varphi,\psi}(t,r):=\varphi\Big(t_\varphi\big(s_\psi(t)\big)\Big)\frac{h-r}{h}+\psi\Big(t_\psi\big(s_\psi(t)\big)\Big)\frac{r}{h}=\overline \varphi(s_\psi(t))\frac{h-r}{h}+\overline \psi(s_\psi(t))\frac{r}{h},
\end{align}
where we recall $\overline \varphi(s)=\varphi\circ t_\varphi(s)$ and $\overline \psi(s)=\psi\circ t_\psi(s)$. This satisfies $\Psi_{\varphi,\psi}(t,0)=\overline \varphi(s_\psi(t))=\Phi_{\varphi,\psi}(t,0)$ and $\Psi_{\varphi,\psi}(t,h)=\overline \psi(s_\psi(t))=\psi(t)$.
We compute the derivatives of $\Phi_{\varphi,\psi}$ and $\Psi_{\varphi,\psi}$ and for a.e. $(t,r)\in [a,b]\times [0,h]$ we find
\begin{align*}
	\frac{\partial }{\partial t}\Phi_{\varphi,\psi}(t,r)&=\dot \varphi\Big(t_\varphi\big(s_\varphi(t)\frac{r}{h}+s_\psi(t)\frac{h-r}{h} \big)\Big)\dot t_\varphi\big(s_\varphi(t)\frac{r}{h}+s_\psi(t)\frac{h-r}{h} \big)\big(\dot s_\varphi(t)\frac{r}{h}+\dot s_\psi(t)\frac{h-r}{h}\big),\\
		\frac{\partial }{\partial r}\Phi_{\varphi,\psi}(t,r)&=\dot \varphi\Big(t_\varphi\big(s_\varphi(t)\frac{r}{h}+s_\psi(t)\frac{h-r}{h} \big)\Big)\dot t_\varphi\big(s_\varphi(t)\frac{r}{h}+s_\psi(t)\frac{h-r}{h} \big)\frac{s_\varphi(t)-s_\psi(t)}{h},\\
		\frac{\partial }{\partial t}\Psi_{\varphi,\psi}(t,r)&=\left(\dot \varphi\Big(t_\varphi\big(s_\psi(t)\big)\Big)\dot t_\varphi\big(s_\psi(t)\big)\frac{h-r}{h}+\dot \psi\Big(t_\psi\big(s_\psi(t)\big)\Big)\dot t_\psi\big(s_\psi(t)\big)\frac{r}{h}\right)\dot s_\psi(t)\\
		&=\left(\frac{h-r}{h}\dot{\overline \varphi}(s_\psi(t))+\frac{r}{h}\dot{\overline \psi}(s_\psi(t))\right)\dot s_\psi(t),\\
		\frac{\partial }{\partial r}\Psi_{\varphi,\psi}(t,r)&=\frac1h\left(\psi\Big(t_\psi\big(s_\psi(t)\big)\Big)-\varphi\Big(t_\varphi\big(s_\psi(t)\big)\Big)\right)=\frac{{\overline \psi}(s_\psi(t))-{\overline \varphi}(s_\psi(t))}{h},
\end{align*}
which, by \eqref{s_varphi} and \eqref{t_varphi}, lead to the following estimates
\begin{align*}
\left|\frac{\partial }{\partial t}\Phi_{\varphi,\psi}(t,r)\right|&\leq (L_\varphi+(b-a))\left|\dot s_\varphi(t)\frac{r}{h}+\dot s_\psi(t)\frac{h-r}{h}\right|\\
&\leq (L_\varphi+(b-a))\left(\frac{|\dot \varphi(t)|+1}{L_\varphi+(b-a)}+\frac{|\dot \psi(t)|+1}{L_\psi+(b-a)}\right),\\
\left|\frac{\partial }{\partial r}\Phi_{\varphi,\psi}(t,r)\right|&\leq\frac{L_\varphi+(b-a)}{h}{|s_\psi(t)-s_\varphi(t)|};
\end{align*}
furthermore we also have
\begin{align}\label{stimadetPhi}
	&\frac{\partial }{\partial t}\Phi_{\varphi,\psi}(t,r)\wedge 	\frac{\partial }{\partial r}\Phi_{\varphi,\psi}(t,r)=\det\left(\nabla \Phi_{\varphi,\psi}(t,r) \right)=0,
\end{align}
for almost every $(t,r)\in [a,b]\times [0,h]$, due to the fact that the image of $\Phi_{\varphi,\psi}$ is one dimensional.
Finally we can estimate on $D:=[a,b]\times [0,h]$ the integral
\begin{align}\label{stimanablaPhi}
\int_D|\nabla \Phi_{\varphi,\psi}(t,r)|dtdr&\leq (L_\varphi+(b-a)) \int_D\frac{|\dot \varphi(t)|+1}{L_\varphi+(b-a)}+\frac{|\dot \psi(t)|+1}{L_\psi+(b-a)}+\frac{|s_\psi(t)-s_\varphi(t)|}{h}dtdr\nonumber\\
&= 2h(L_\varphi+(b-a))+(L_\varphi+(b-a))\int_a^b|s_\psi(t)-s_\varphi(t)|dt.
\end{align}
As for $\Psi_{\varphi,\psi}$, by the estimates
\begin{align*}
	&\left|	\frac{\partial }{\partial t}\Psi_{\varphi,\psi}(t,r)\right|= \left|\frac{h-r}{h}\dot{\overline \varphi}(s_\psi(t))+\frac{r}{h}\dot{\overline \psi}(s_\psi(t))\right|\dot s_\psi(t)\leq (L_\varphi+L_\psi+(b-a))\dot s_\psi(t),\\
	&\left|	\frac{\partial }{\partial r}\Psi_{\varphi,\psi}(t,r)\right|\leq \frac{|{\overline \psi}(s_\psi(t))-{\overline \varphi}(s_\psi(t))|}{h},
\end{align*}
we can write
\begin{align}\label{stimanablaPsi}
\int_D|\nabla \Psi_{\varphi,\psi}(t,r)|dtdr&\leq \int_D(L_\varphi+L_\psi+(b-a))\dot s_\psi(t)+\frac{|{\overline \psi}(s_\psi(t))-{\overline \varphi}(s_\psi(t))|}{h}dtdr\nonumber\\
&= (L_\varphi+L_\psi+(b-a))h+\int_a^b|{\overline \psi}(s_\psi(t))-{\overline \varphi}(s_\psi(t))|\dot s_\psi(t)dt\nonumber\\
&=(L_\varphi+L_\psi+(b-a))h+\int_0^{1}|{\overline \psi}(s)-{\overline \varphi}(s)|ds,
\end{align}
where we have used that $\int_a^b\dot s_\psi(t)dt=1$.
Finally
\begin{align}\label{stimadetPsi}
	 \nonumber\int_D|\frac{\partial }{\partial t}\Psi_{\varphi,\psi}(t,r)\wedge \frac{\partial }{\partial r}\Psi_{\varphi,\psi}(t,r)|dtdr&\leq(L_\varphi+L_\psi+(b-a))\int_D\frac{|{\overline \psi}(s_\psi(t))-{\overline \varphi}(s_\psi(t))|}{h}\dot s_\psi(t)dtdr\\
	&\nonumber= (L_\varphi+L_\psi+(b-a))\int_a^b|{\overline \psi}(s_\psi(t))-{\overline \varphi}(s_\psi(t))|\dot s_\psi(t)dt\\&=(L_\varphi+L_\psi+(b-a))\int_0^{1}|{\overline \psi}(s)-{\overline \varphi}(s)|ds.
\end{align}

\subsection{Tubular neighborhoods of regular curves}

Given a set $A\subset \R^2$ we denote by $\text{{\rm dist}}(x, A)$ the distance from $x$ to $A$, and by $\text{{\rm dist}}^\pm(x,A)$ the signed distance from $x$ to $A$, defined as 
$$\text{{\rm dist}}^\pm(x,A):=\begin{cases}
	\text{{\rm dist}}(x,A)&\text{if }x\in A^c,\\
	-\text{{\rm dist}}(x,A^c)&\text{if }x\in A,
\end{cases}$$
where $A^c:=\R^2\setminus A$. 
We consider the following regularity assumption (R) of a set $A$:
\begin{itemize}
\item[(R)] We assume that $A$ is a connected bounded open set with boundary of class $C^3$.
\end{itemize}
If $A\subset\R^2$ satisfies (R), then $\partial A$ consists of finitely many loops $\Gamma_i$, $i=0,1,\dots,N$, of class $C^3$, labeled so that, if $E_i$ denotes the bounded connected component of $\R^2\setminus \Gamma_i$, then
\begin{align}\label{deca}
	A=E_0\setminus (\cup_{i=1}^NE_i).
\end{align}
Notice that the presence of a unique big component $E_0$ is due to the hypothesis that $A$ is connected\footnote{If $A$ instead has $K>1$ connected components, then every component enjoys a decomposition as \eqref{deca}.}.

\textbf{Sets with $C^3$-boundary and tubular neighborhoods}: 
Let $A\subset\R^2$ be a set satisfying (R). 
 For $\delta\in(0,1)$ small enough there exists a tubular neighborhood $T_\delta$ of $\partial A$, given by 
$$T_\delta:=\{x\in \R^2:\text{{\rm dist}}(x,\partial A)<\delta\}.$$
We parametrize $T_\delta$ with $(t,r)\in \partial A\times (-\delta,\delta)$ so that 
$$\partial A_r:=\{x\in \R^2:\text{{\rm dist}}^\pm(x, A)=r\}$$
consists of $N+1$ curves $\Gamma_r^i$ of class $C^2$, namely
$$\Gamma_r^0:=\{x\in \R^2:\text{{\rm dist}}^\pm(x,E_0)=r\}\qquad \qquad \Gamma_r^i:=\{x\in \R^2:\text{{\rm dist}}^\pm(x,E_i)=-r\}.$$ 
We denote $T_\delta=\cup_{i=1}^NT^i_\delta $ where $T^i_\delta $ is a $\delta$-neighborhood of $\Gamma_i$, namely $$T^i_\delta=\{x\in \R^2:\text{{\rm dist}}(x,\Gamma_i)<\delta\}.$$

For simplicity\footnote{The following argument applies to all connected components of $\Gamma$ in the general case.}, let us assume that the number $N$ of holes in $A$ is zero, i.e., $A$ is simply connected; there is  $\gamma\in C^3([a,b];\R^2)$ a Jordan curve parametrized by arc-length enclosing the open bounded connected and simply-connected set $A$, $\Gamma=\gamma([a,b])$.
We will denote
 $$T^+_\delta=\{x\in \R^2:\text{{\rm dist}}^\pm(x,A)\in (0,\delta)\},\qquad \qquad T^-_\delta=\{x\in \R^2:\text{{\rm dist}}^\pm(x,A)\in (-\delta,0)\},$$
 the external and inner tubular neighborhoods of $\partial A$.
By the tubular neighborhood theorem, there exists a bi-Lipschitz bijection  $\mathcal T_\delta:[a,b)\times (-\delta,\delta)\rightarrow T_\delta$,  such that 
$$|\det(\nabla \mathcal T_\delta(t,r))|=1+R_\delta(t,r),$$
where $\|R_\delta\|_{L^\infty}=o(1)\rightarrow 0$ as $\delta\rightarrow 0$.
Indeed one sets, for all $(t,r)\in [a,b)\times (-\delta,\delta)$,
\begin{align}\label{mathcalTdelta}
	\mathcal T_\delta(t,r):=\gamma(t)+r\dot\gamma(t)^\perp,
\end{align}
where $v^\perp=(-v_2,v_1)$, and it holds 
\begin{align*}
	&\frac{\partial}{\partial t}\mathcal T_\delta(t,r)=\dot\gamma(t)+r\ddot\gamma(t)^\perp,\qquad\qquad\qquad\qquad\qquad\frac{\partial}{\partial r}\mathcal T_\delta(t,r)=\dot\gamma(t)^\perp,\\
	&\det(\nabla \mathcal T_\delta)=1+r\dot\gamma(t)\cdot \ddot \gamma(t)^\perp=:1+R_\delta(t,r),\qquad \qquad |R_\delta(t,r)|\leq C_\gamma|r|\leq C_\gamma\delta,
\end{align*}
where, here and below, we denote by $C_\gamma>0$ a constant  depending  on $\gamma$ but independent of $\delta$ (and which might change from line to line).
Notice also that since $\gamma$ is of class $C^3$, $\nabla  \mathcal T_\delta$ is of class $C^1$,  and (since $\delta\in(0,1)$)
\begin{align*}
	|\nabla \mathcal T_\delta(t,r)|\leq |\dot\gamma(t)|+r|\ddot\gamma(t)|\leq C_\gamma+C_\gamma\delta\leq C_\gamma,
\end{align*}
Let $h\in (0,\delta)$. For $x\in T_\delta$ we have
$\nabla \mathcal T_h^{-1}(x)=\big(\nabla \mathcal T_h(\mathcal T_h^{-1}(x))\big)^{-1}$, so  $$\det(\nabla \mathcal T_h^{-1}(x))=\frac{1}{\det\big(\nabla \mathcal T_h(\mathcal T_h^{-1}(x))\big)}
=\frac{1}{1+R_h(\mathcal T_h^{-1}(x))}=1-\frac{R_h(\mathcal T_h^{-1}(x))}{1+R_h(\mathcal T_h^{-1}(x))},$$
and, if $h$ is small enough, we conclude 
\begin{align}\det(\nabla \mathcal T_h^{-1}(x))= 1+R'_h(x),\qquad \qquad \|R'_h\|_{L^\infty}\leq C_\gamma h.\label{detT-1}
\end{align}
Eventually, using that for a invertible matrix $A$ one has $A^{-1}=\cof(A)^T(\det A)^{-1}$, we conclude
\begin{align}
	&\nabla \mathcal T_h^{-1}(x)= \cof\big(\nabla \mathcal T_h(\mathcal T_h^{-1}(x))\big)^T(1+R'_h(x)),\nonumber\\
	&|\nabla \mathcal T_h^{-1}(x)|\leq C_\gamma+C_\gamma h\leq C_\gamma,
\end{align}
so $\mathcal T_h$ is bi-Lipschitz with a constant depending only on $\gamma$.
\medskip

\textbf{Restriction of BV-functions on curves:}
As above, let $A$ satisfy (R), assume that $A$ is simply connected, and let  $\gamma\in C^3([a,b];\R^2)$ be an arc-length parametrization of a Jordan curve $\Gamma=\partial A$. Let $T_\delta$ be a tubular neighborhood of $\Gamma$, $\delta\in(0,1)$ small enough.
 Let $\widehat \zeta:[a,b]\times(-\delta,\delta)\rightarrow\R^2$ be the map
\begin{align}\label{zetahat}
\widehat \zeta(t,r):=\frac{\frac{\partial \mathcal T_\delta}{\partial t}(t,r)}{|\frac{\partial \mathcal T_\delta}{\partial t}(t,r)|}=\frac{\dot\gamma(t)+r\ddot\gamma(t)^\perp}{|\dot\gamma(t)+r\ddot\gamma(t)^\perp|},
\end{align}
that is the oriented unit vector tangent to $\Gamma_{r}$ at the point $\gamma(t,r)$. Using that $\gamma$ parametrizes by arc-length, a tedious but straightforward computation shows that the map
\begin{align}
	\zeta(x):=\widehat\zeta(\mathcal T_\delta^{-1}(x)),\qquad \qquad x\in T_\delta,\label{zeta}
\end{align}
satisfies  $\zeta\in C^1(T_\delta;\mathbb S^1)$ and is  divergence free\footnote{We can also see this as follows: $\zeta$ is a unit vector such that $\zeta^\perp$ is orthogonal to the level sets of the signed distance function $d^\pm$ from $\Gamma$. In particular, since the distance function has gradient of length $1$ almost everywhere, $\zeta^\perp$ coincides with $\nabla d^\pm$ almost everywhere. If follows that $\div \zeta=\curl \zeta^\perp=\curl \nabla d^\pm=0$.}.
\begin{definition}\label{defc1}
Let $r\in (-\delta,\delta)$ and  $\varphi:\Gamma_r\rightarrow \R^m$; we say that $\varphi\in C^1(\Gamma_{r};\R^m)$ if $\varphi(\mathcal T_\delta(\cdot,r)):[a,b)\rightarrow \R^m$ is of class $C^1$.
\end{definition}

\begin{remark}\label{rem_39}
Given $\varphi\in C^1(\Gamma_{r};\R^m)$ we can extend it on $T_\delta$ by defining $\overline \varphi(t,r'):=\varphi(\gamma(t)+r\dot\gamma(t)^\perp)$ for all $r'\in (-\delta,\delta)$ and $t\in [a,b)$. The function $\overline \varphi\circ \mathcal T_\delta^{-1}(x)$ defined for all $x\in T_\delta$ is then an extension of $\varphi$ and is of class $C^1$. Indeed, clearly $\overline \varphi\in C^1([a,b)\times (-\delta,\delta))$, and so $\overline \varphi\circ \mathcal T_\delta^{-1}\in C^1(T_\delta)$ because $\mathcal T_\delta^{-1}$ is of class $C^1$. In particular, we conclude that every function $\varphi\in C^1(\Gamma_{r};\R^m)$ is the restriction to $\Gamma_r$ of a function of class $C^1(T_\delta;\R^m)$. Since it is also easy to see that every function of class $C^1(T_\delta;\R^m)$ has a $C^1$ restriction on $\Gamma_r$ as in Definition \ref{defc1}, 
we conclude that $\varphi \in C^1(\Gamma_{r};\R^m)$ if and only if it is the restriction of a function $\widehat \varphi\in C^1(T_\delta;\R^m)$ on $\Gamma_r$.
\end{remark}

\begin{definition}\label{tang_der}
Let $u:\Gamma_r\rightarrow \R^m$, we say that $u\in BV(\Gamma_{r};\R^m)$ if 
$$\sup\{\int_{\Gamma_r}u\cdot\Big(\sum_{j=1}^2D_j(\varphi \zeta_j)\Big)d\mathcal H^1:\varphi\in C^1(T_\delta;\R^m),|\varphi|\leq 1\}<+\infty.$$
We denote the supremum above by $|D_\zeta u|(\Gamma_r)$.
\end{definition}

Exploiting that $\zeta$ is divergence-free, we can write
$$|D_\zeta u|(\Gamma_{r})=\sup\{\int_{\Gamma_r}u\cdot D_\zeta\varphi d\mathcal H^1:\varphi\in C^1(T_\delta;\R^m),|\varphi|\leq 1\},$$
where $D_\zeta\varphi:=\sum_{j=1}^2D_j\varphi\zeta_j$.
Recalling that $\mathcal T_\delta(\cdot,r)$ is a parametrization of $\Gamma_r$, if $u\in BV(\Gamma_r;\R^m)$ we see that 
\begin{align*}
\int_{a}^b|\frac{d}{dt}u(\mathcal T_\delta(t,r))|dt&=\sup\{\int_a^b\frac{d}{dt}u(\mathcal T_\delta(t,r))\cdot \psi(\mathcal T_\delta(t,r))dt:\psi\in C^1(\Gamma_r;\R^m), |\psi|\leq 1\}\\
&=\sup\{\int_a^bu(\mathcal T_\delta(t,r))\cdot \frac{d}{dt}\psi(\mathcal T_\delta(t,r))dt:\psi\in C^1(\Gamma_r;\R^m), |\psi|\leq 1\}
\end{align*}
and, up to  extending $\psi$ to  $T_\delta$ as in Remark \ref{rem_39}, we have
$$\frac{d}{dt}\psi(\mathcal T_\delta(t,r))=\nabla \psi(\mathcal T_\delta(t,r))\frac{\partial \mathcal T_\delta}{\partial t}(t,r)=\nabla \psi(\mathcal T_\delta(t,r))\widehat\zeta(t,r)|\frac{\partial \mathcal T_\delta}{\partial t}(t,r)|,$$
so we conclude
\begin{align}\label{Dugammar}
	\int_{a}^b|\frac{d}{dt}u(\mathcal T_\delta(t,r))|dt
	&=\sup\{\int_{\Gamma_r}u\cdot D_\zeta\psi d\mathcal H^1:\psi\in C^1(\Gamma_r;\R^m), |\psi|\leq 1\}=|D_\zeta u|(\Gamma_{r}).
\end{align}

\begin{remark}\label{rem_stricttraces}
Equality \eqref{Dugammar} in particular implies that 
if $u_k,u\in BV(\Gamma_r;\R^m)$ are such that 
$$u_k\rightarrow u\qquad \qquad \text{ strictly in }BV(\Gamma_r;\R^m),$$
then also
$$u_k(\mathcal T_\delta(\cdot,r))\rightarrow u(\mathcal T_\delta(\cdot,r))\qquad \qquad \text{ strictly in }BV([a,b];\R^m),$$
and viceversa. More precisely, for all $r\in(-\delta,\delta)$ and any $v\in BV(\Gamma_r;\R^m)$ it holds
$$|D_\zeta v|(\Gamma_r)=|D_t(v\circ \mathcal T_\delta(\cdot,r))|(a,b),$$
and there are two positive constants $c_\delta,C_\delta$ depending only on $\Gamma$ and $\delta$ such that 
$$c_\delta \|u\circ\mathcal T_\delta(\cdot,r)\|_{L^1([a,b])}\leq \|u\|_{L^1(\Gamma_r)}\leq C_\delta \|u\circ\mathcal T_\delta(\cdot,r)\|_{L^1([a,b])}.$$
This follows from the bi-lipschitz property of $\mathcal T_\delta$ and on the fact that $|\frac{d}{dt} T_\delta(\cdot,r)|$ is close to $1$, for $r\in (-\delta,\delta)$.
\end{remark}
Given $v:T_\delta\rightarrow \R^m$ a Lipschitz map, then by coarea formula we can write
$$\int_{T_\delta}|\nabla v\zeta|dx=\int_{-\delta}^\delta\int_{\Gamma_r}|\nabla v\zeta|d\mathcal H^1dr=\int_{-\delta}^\delta\int_{\Gamma_r}|D_\zeta v|d\mathcal H^1dr,$$
and since $\zeta$ is a unit oriented tangent vector to $\Gamma_r$, $\nabla v\zeta=\sum_{j=1}^2D_jv\zeta_j$ represents the tangential derivative $D_\zeta v$ of $v$ to $\Gamma_r$. Now, $\mathcal T_\delta(\cdot,r)$ is a parametrization from $[a,b]$ of $\Gamma_r$, so we write
\begin{align}
&\int_{a}^b|\frac{d}{dt}v(\mathcal T_\delta(t,r))|dt=\int_a^b|\nabla v(\mathcal T_\delta(t,r))\frac{d\mathcal T_\delta}{dt}(t,r)|dt\nonumber\\
&=\int_a^b|\nabla v(\mathcal T_\delta(t,r))\zeta(t,r)||\frac{d\mathcal T_\delta}{dt}(t,r)|dt=\int_{\Gamma_r}|D_\zeta v|d\mathcal H^1,
\end{align}
and we conclude
\begin{align}
\int_{T_\delta}|D_\zeta v|dx=\int_{-\delta}^\delta\int_{a}^b|\frac{d}{dt}v(\mathcal T_\delta(t,r))|dtdr.
\end{align}

In the following lemma we discuss how strict convergence is inehrited on curves. 

\begin{lemma}\label{lemma_coareaBV}
Let $u_k:T_\delta\rightarrow \R^m$ be Lipschitz maps and let $u\in BV(T_\delta;\R^m)$ be such that 
$$u_k\rightarrow u\qquad \text{strictly in }BV(T_\delta;\R^m).$$
Then, for  a.e. $r\in(-\delta,\delta)$ the function $u\res \Gamma_r$ belongs to $BV(\Gamma_r;\R^m)$ and (up to a non-relabelled subsequence) $u_k\res\Gamma_r$ converge strictly in $BV(\Gamma_r;\R^m)$ to $u\res \Gamma_r$.
\end{lemma}
\begin{proof}

By Reshetniak Theorem \ref{teo_32} we have, as $k\rightarrow\infty$,
\begin{align}\label{resh}
	\int_{T_\delta}|D_\zeta u_k|dx=
\int_{T_\delta}|\nabla u_k\zeta|dx\rightarrow \int_{T_\delta}\Big|\frac{Du}{|Du|} \zeta\Big|d|Du|.
\end{align}
The quantity in the right-hand side is equal to
\begin{align*}
\int_{T_\delta}\Big|\frac{Du}{|Du|} \zeta\Big|d|Du|&=\sup\{\int_{T_\delta}\sum_{j=1}^2\varphi\cdot\frac{D_ju}{|Du|}\zeta_j  d|Du|:\varphi\in C^1(T_\delta;\R^m),|\varphi|\leq 1\}\\
&=\sup\{\int_{T_\delta}\sum_{j=1}^2\zeta_j \varphi\cdot dD_ju:\varphi\in C^1(T_\delta;\R^m),|\varphi|\leq 1\}\\
	&=\sup\{\int_{T_\delta}u\cdot (\nabla\varphi\zeta) dx:\varphi\in C^1(T_\delta;\R^m),|\varphi|\leq 1\}
\end{align*}
where in the last equality we have used the divergence-free property of $\zeta$.
Therefore, by \eqref{resh}, we conclude
\begin{align}
\lim_{k\rightarrow \infty}\int_{T_\delta}|\nabla u_k\zeta|dx=\sup\{\int_{T_\delta}u \cdot D_\zeta\varphi dx:\varphi\in C^1(T_\delta;\R^m),|\varphi|\leq 1\}.
\end{align}
On the other hand 
\begin{align*}
\int_{T_\delta}|\nabla u_k\zeta|dx=\int_{-\delta}^\delta\int_{a}^b|\frac{d}{dt}u_k(\mathcal T_\delta(t,r))|dtdr,
\end{align*}
whence
\begin{align}\label{STRICT}
\lim_{k\rightarrow \infty}\int_{-\delta}^\delta\int_{a}^b|\frac{d}{dt}u_k(\mathcal T_\delta(t,r))|dtdr=\sup\{\int_{T_\delta}u\cdot D_\zeta\varphi dx:\varphi\in C^1(T_\delta;\R^m),|\varphi|\leq 1\}.
\end{align}
Now, by Fatou Lemma
\begin{align}\label{fatou_strict}
	\lim_{k\rightarrow \infty}\int_{-\delta}^\delta\int_{a}^b|\frac{d}{dt}u_k(\mathcal T_\delta(t,r))|dtdr\geq \int_{-\delta}^\delta\liminf_{k\rightarrow \infty}\int_{a}^b|\frac{d}{dt}u_k(\mathcal T_\delta(t,r))|dtdr
\end{align}
and we know from the strict convergence of $u_k$ to $u$ that for a.e. $r\in (-\delta,\delta)$ the trace $u_k\res\Gamma_r$ converges to $u\res\Gamma_r$ in $L^1(\Gamma_r;\R^m)$. This implies that, for a.e. $r\in (-\delta,\delta)$ 
\begin{align}\label{eq_aer}
\liminf_{k\rightarrow \infty}\int_{a}^b|\frac{d}{dt}u_k(\mathcal T_\delta(t,r))|dt&\geq \int_{a}^b|\frac{d}{dt}u(\mathcal T_\delta(t,r))|dt\nonumber\\
&=\sup\{\int_{\Gamma_r}u\cdot D_\zeta\varphi d\mathcal H^1:\varphi\in C^1(T_\delta;\R^m),|\varphi|\leq 1\}
\end{align}
where we have used \eqref{Dugammar}; so that 
\begin{align}\label{fatou_strict2}
\int_{-\delta}^\delta\liminf_{k\rightarrow \infty}\int_{a}^b|\frac{d}{dt}u_k(\mathcal T_\delta(t,r))|dtdr&\geq \int_{-\delta}^\delta \sup\{\int_{\Gamma_r}u\cdot\nabla\varphi\zeta d\mathcal H^1:\varphi\in C^1(T_\delta;\R^m),|\varphi|\leq 1\}dr\nonumber\\
&\geq \sup \{\int_{-\delta}^\delta\int_{\Gamma_r}u\cdot\nabla\varphi\zeta d\mathcal H^1dr:\varphi\in C^1(T_\delta;\R^m),|\varphi|\leq 1\}.
\end{align}
We have found then, from \eqref{STRICT}, that the inequalities in \eqref{fatou_strict}  and \eqref{fatou_strict2} are all equalities. In particular, equality in \eqref{eq_aer} holds for a.e. $r\in (-\delta,\delta)$, and denoting 
$$f(r):=\int_{a}^b|\frac{d}{dt}u(\mathcal T_\delta(t,r))|dt\qquad \qquad f_k(r):=\int_{a}^b|\frac{d}{dt}u_k(\mathcal T_\delta(t,r))|dt$$ 
equality  \eqref{fatou_strict} implies that 
$$\lim_{k\rightarrow \infty}\int_{-\delta}^\delta f_k(r)dr=\int_{-\delta}^\delta f(r)dr,\qquad \qquad \liminf_{k\rightarrow\infty} f_k(r)=f(r).$$
Thus Lemma \ref{Fatou=} in the Appendix entails that $f_k\rightarrow f$ in $L^1((-\delta,\delta))$, and there is a subsequence such that  for a.e. $r\in (-\delta,\delta)$
$$f_k(r)\rightarrow f(r),$$
that is the thesis.
\end{proof}

%
%
%
%
%
%
%
%
%
%
%
%
%

 \textbf{Transformations in tubular neighborhoods}:
Let $\Gamma:=\gamma([a,b])$ be a Jordan curve parametrized by arc-length by $\gamma\in C^3([a,b];\R^2)$, and  enclosing the  simply-connected set $A$ satisfying (R); let $\delta\in(0,1)$ be small enough and let $T_\delta$ be a tubular neighborhood of $\Gamma$. We want to define a bijection between $T_\delta$ and itself, which will be needed to modify suitable recovery sequences $u_k$ for the involved functional. To this aim, we first introduce for $c\in (0,\delta)$ fixed, and $n\in \mathbb N$, $n>\frac2\delta$,  the map
$$\Upsilon_{\delta,n,c}:[a,b]\times [-\delta,\delta]\rightarrow[a,b]\times [-\delta,\delta],\qquad\qquad \Upsilon_{\delta,n,c}(t,r)=(t,\tau_{\delta,n,c}(r)),$$
where $\tau_{\delta,n,c}$ is the piecewise affine interpolant such that $\tau_{\delta,n,c}(-\delta)=-\delta$, $\tau_{\delta,n,c}(-\frac{c}{n})=0$, and $\tau_{\delta,n,c}(\delta)=\delta$, namely
$$\tau_{\delta,n,c}(r)=\begin{cases}
						\frac{n\delta r+c\delta}{n\delta -c}&\text{for }r\in[-\delta,-\frac cn),\\
							\frac{n\delta r+c\delta}{n\delta+c}&\text{for }r\in[-\frac cn,\delta].
					\end{cases}$$
For all $(t,s)\in [a,b]\times [-\delta,\delta]$ we write
\begin{align}\label{Y}
\Upsilon_{\delta,n,c}(t,s)=(t,s)+(0,\tau_{\delta,n,c}(s)-s),\qquad \qquad\text{with } |(0,\tau_{\delta,n,c}(s)-s)|\leq \frac{C}{n},
\end{align}	
for a constant $C>0$ independent of $\delta$ and $n>\frac2\delta$.
Computing $\nabla \Upsilon_{\delta,n,c}$, we write
\begin{align}\label{nablaY}
\nabla \Upsilon_{\delta,n,c}=\text{{\rm Id}}+M_{\delta,n,c},\qquad \qquad M_{\delta,n,c}:=\begin{pmatrix} 0&0\\0&\dot\tau_{\delta,n,c}-1
\end{pmatrix},
\end{align}
in such a way that $|M_{\delta,n,c}|\leq \frac{C}{n}$ (here  $C$ is a positive constant independent of $n>\frac2\delta$ and $\delta$). Analogously, it is immediately checked that
			\begin{align}\label{nablaYinverse}
				\nabla \Upsilon^{-1}_{\delta,n,c}=\text{{\rm Id}}+M'_{\delta,n,c},\qquad \qquad
				\text{with }|M'_{\delta,n,c}|\leq \frac{C}{n},
			\end{align}	
and for all $(t,s)\in [a,b]\times [-\delta,\delta]$ we have $\Upsilon_{\delta,n,c}^{-1}(t,s)=(t,\tau_{\delta,n,c}^{-1}(s))$, so we may write
\begin{align}\label{Yinverse}
\Upsilon_{\delta,n,c}^{-1}(t,s)=(t,s)+(0,\tau_{\delta,n,c}^{-1}(s)-s),\qquad \qquad\text{with } |(0,\tau_{\delta,n,c}^{-1}(s)-s)|\leq \frac{C}{n}.
\end{align}	
We now define, for  $\delta\in(0,1)$ as above and $n\in \mathbb N$, $n>\frac2\delta$, the following transformation
\begin{align}
\Sigma_{\delta,n,c}:\overline T_\delta\rightarrow\overline  T_\delta,\qquad \qquad \Sigma_{\delta,n,c}:=\mathcal T_\delta \circ \Upsilon_{\delta,n,c}\circ \mathcal T_\delta^{-1}.
\end{align}
This map sends the set $\mathcal T_\delta([a,b],-\frac cn)$ to the curve $\Gamma$. 
Moreover there is a constant $C_\gamma$, depending only on $\gamma$, such that 
\begin{align}\label{Sigma_uniform}
|\Sigma_{\delta,n,c}(x)-x|\leq \frac{C_\gamma}{n}, \qquad \qquad \forall x\in \overline T_\delta.
\end{align}
This follows from \eqref{Y} and the Lipschitz continuity of $\mathcal T_\delta$. 
It is convenient also to introduce 
\begin{align}\label{Sigma-}
\Sigma_{\delta,n,c}^-:\overline T_\delta^-\setminus T_{\frac cn}\rightarrow \overline T^-_\delta,\qquad \qquad \Sigma^-_{\delta,n,c}:=(\mathcal T_\delta \circ \Upsilon_{\delta,n,c}\circ \mathcal T_\delta^{-1})\res (\overline T^-_\delta\setminus T_{\frac cn}),
\end{align}
the restriction of $\Sigma_{\delta,n,c}$ to  $\overline T^-_\delta\setminus T_{\frac cn}$.
For all $x\in T_\delta$, we have
\begin{align}\label{nablaSigma}
\nabla \Sigma_{\delta,n,c}(x)=\nabla \mathcal T_\delta(\Upsilon_{\delta,n,c}\circ \mathcal T_\delta^{-1}(x))\nabla \Upsilon_{\delta,n,c}(\mathcal T_\delta^{-1}(x))\nabla \mathcal T_\delta^{-1}(x),
\end{align}
and writing $\nabla\mathcal T_\delta(\Upsilon_{\delta,n,c}\circ \mathcal T_\delta^{-1}(x))=\nabla\mathcal T_\delta\Big(\mathcal T_\delta^{-1}(x)+(\Upsilon_{\delta,n,c}\circ \mathcal T_\delta^{-1}(x)-\mathcal T_\delta^{-1}(x))\Big)$, we get
\begin{align}\label{335}
\nabla\mathcal T_\delta(\Upsilon_{\delta,n,c}\circ \mathcal T_\delta^{-1}(x))=\nabla\mathcal T_\delta(\mathcal T_\delta^{-1}(x))+\rho_{\delta,n,c}(x),
\end{align}
where, by using the Lipschitz continuity  of $\nabla \mathcal T_\delta$ (it is of class $C^1$) and by \eqref{Y}, the matrix $$\rho_{\delta,n,c}(x):=\nabla\mathcal T_\delta\Big(\mathcal T_\delta^{-1}(x)+(\Upsilon_{\delta,n,c}\circ \mathcal T_\delta^{-1}(x)-\mathcal T_\delta^{-1}(x))\Big)-\nabla\mathcal T_\delta(\mathcal T_\delta^{-1}(x))$$ enjoies
\begin{align}\label{212}
	\left|\rho_{\delta,n,c}(x)\right|\leq \frac{C_\gamma}{n}
\end{align}
(here and below, unless explicitely stated,  $C_\gamma$ is a positive constant independent of $n>\frac2\delta$ and $\delta$, but depending on $\gamma$).
Plugging \eqref{nablaY} and \eqref{335} into \eqref{nablaSigma} we obtain
\begin{align}\label{nablaSigma2}
\nabla \Sigma_{\delta,n,c}(x)&=(\nabla\mathcal T_\delta(\mathcal T_\delta^{-1}(x))+\rho_{\delta,n,c}(x))(\text{{\rm Id}}+M_{\delta,n,c}(\mathcal T_\delta^{-1}(x)))\nabla\mathcal T_\delta^{-1}(x)\nonumber\\
&=\text{{\rm Id}}+\nabla\mathcal T_\delta(\mathcal T_\delta^{-1}(x))M_{\delta,n,c}(\mathcal T_\delta^{-1}(x))\nabla\mathcal T_\delta^{-1}(x)+\rho_{\delta,n,c}(x)(\text{{\rm Id}}+M_{\delta,n,c}(\mathcal T_\delta^{-1}(x)))\nabla\mathcal T_\delta^{-1}(x)\nonumber\\
&=:\text{{\rm Id}}+\sigma_{\delta,n,c}(x),
\end{align}
where we have used that $\nabla\mathcal T_\delta(\mathcal T_\delta^{-1}(x))=(\nabla\mathcal T_\delta^{-1}(x))^{-1}$ and, thanks to \eqref{nablaY}, \eqref{212}, and the Lipschitz continuity of $\nabla \mathcal T_\delta$, we have
\begin{align}\label{gradSigma}
	|\sigma_{\delta,n,c}(x)|\leq \frac{C_\gamma}{n}.
\end{align}
Finally, by \eqref{nablaSigma2}, we have also, for $n$ large enough
\begin{align}\label{detSigma}
\det(\nabla \Sigma_{\delta,n,c}(x))=1+d_{\delta,n,c}(x),\qquad \qquad \text{with }\|d_{\delta,n,c}\|_{L^\infty}\leq \frac{C_\gamma}{n},
\end{align}
and a similar expression holds for $\det(\nabla \Sigma_{\delta,n,c}(x)^{-1})$, namely
\begin{align}\label{detSigmainverse}
	\det(\nabla \Sigma_{\delta,n,c}(x)^{-1})=1+\widehat d_{\delta,n,c}(x),\qquad \qquad \text{with }\|\widehat d_{\delta,n,c}\|_{L^\infty}\leq \frac{C_\gamma}{n}.
\end{align}
In what follows we will sometimes employ also the map $\widehat \Sigma_{\delta,n,c}$ that is defined as $\Sigma_{\delta,n,c}$ but with $\mathcal T_\delta$ replaced by $\widehat {\mathcal T}_\delta$ given by
$$\widehat {\mathcal T}_\delta(t,r)={\mathcal T}_\delta(t,-r),$$
for all $(t,r)\in [a,b]\times (-\delta,\delta).$
Namely
\begin{align}
	\widehat \Sigma_{\delta,n,c}:T_\delta\rightarrow T_\delta,\qquad \qquad \Sigma_{\delta,n,c}:=\widehat {\mathcal T}_\delta \circ \Upsilon_{\delta,n,c}\circ \widehat{\mathcal T}_\delta^{-1}.
\end{align}
We will consider $ \Sigma_{\delta,n,c}^+:\overline T_\delta^+\setminus T_{\frac cn}\rightarrow \overline T^+_\delta$ defined as
\begin{align}\label{Sigma+}
\Sigma^+_{\delta,n,c}:=(\widehat {\mathcal T}_\delta \circ \Upsilon_{\delta,n,c}\circ \widehat{\mathcal T}_\delta^{-1})\res (\overline T_\delta^+\setminus  T_{\frac cn}).
\end{align} 
For $\widehat \Sigma_{\delta,n,c}$, $\Sigma_{\delta,n,c}^-$, and $\Sigma_{\delta,n,c}^+$ similar estimates as in \eqref{212}, \eqref{gradSigma}, and \eqref{detSigma} hold true. Eventually, using that $\Upsilon_{\delta,n,c}^{-1}$ satisfies \eqref{nablaYinverse} and \eqref{Yinverse}, the same holds also for $\widehat \Sigma_{\delta,n,c}^{-1}$, $\widehat {\mathcal T}_\delta^{-1}$, $(\Sigma_{\delta,n,c}^-)^{-1}$, and $(\Sigma_{\delta,n,c}^+)^{-1}$. Specifically, we will write
\begin{align}\label{nablaSigmapm}
&\nabla \Sigma^\pm_{\delta,n,c}(x)=\text{{\rm Id}}+\sigma^\pm_{\delta,n,c}(x),\qquad \qquad \|\sigma^\pm_{\delta,n,c}\|_{L^\infty}\leq \frac {C_\gamma}{n},\nonumber\\
&\det(\nabla \Sigma^\pm_{\delta,n,c}(x))=1+d^\pm_{\delta,n,c}(x),\qquad \qquad\|d^\pm_{\delta,n,c}\|_{L^\infty}\leq \frac {C_\gamma}{n},\nonumber\\
&\nabla (\Sigma^\pm_{\delta,n,c})^{-1}(x)=\text{{\rm Id}}+\widehat\sigma^\pm_{\delta,n,c}(x),\qquad \qquad \|\widehat\sigma^\pm_{\delta,n,c}\|_{L^\infty}\leq \frac {C_\gamma}{n},\nonumber\\
&\det(\nabla (\Sigma^\pm_{\delta,n,c})^{-1}(x))=1+\widehat d^\pm_{\delta,n,c}(x),\qquad \qquad\|\widehat d^\pm_{\delta,n,c}\|_{L^\infty}\leq \frac {C_\gamma}{n},
\end{align}
where $\sigma_{\delta,n,c}^\pm:\overline T_\delta^\pm\setminus  T_{\frac cn}\rightarrow \R^{2\times 2}$, $d^\pm_{\delta,n,c}:\overline T_\delta^\pm\setminus  T_{\frac cn}\rightarrow \R$, $\widehat \sigma_{\delta,n,c}^\pm:\overline T_\delta^\pm\rightarrow \R^{2\times 2}$, and $\widehat d^\pm_{\delta,n,c}:\overline T_\delta^\pm\rightarrow \R$ are suitable functions.

\subsection{Composition of maps with planar transformations}

In this section we use the planar transformations introduced in the previous section to modify suitable functions defined on planar domains.

\textbf{Interpolations between maps on Jordan curves}:
Let $\Gamma:=\gamma([a,b])$, $\gamma\in C^3([a,b];\R^2)$, be a Jordan curve parametrized by arc-length as in the previous section.
Recalling the functions in \eqref{Phi_interp} and \eqref{Psi_interp}, for two given Lipschitz maps $\varphi,\psi:[a,b]\rightarrow \R^2$ with $\varphi(a)=\varphi(b)$ and $\psi(a)=\psi(b)$, we define the interpolation $H_{\varphi,\psi,h}:\overline T_{h}\rightarrow \R^2$ as
\begin{align}\label{Hinterp}
H_{\varphi,\psi,h}:=\begin{cases}
				 \Phi_{\varphi,\psi}\circ \mathcal T_h^{-1}&\text{ in }\overline T_h^+\\
			 \Psi_{\varphi,\psi}\circ \widehat {\mathcal T}_h^{-1}&\text{ in }\overline T_h^-,
					\end{cases}
\end{align}
where $0<h\leq \delta$ and $\delta\in (0,1)$ is small enough.
The interpolation $H_{\varphi,\psi,h}$ turns out to be Lipschitz continuous.

For $r,s\in (-\delta,\delta)$ fixed, recalling that the map $\mathcal T_\delta(\cdot,r):[a,b]\rightarrow \Gamma_r$ is a parametrization of the curve $\Gamma_r$, it follows that if $u,v$ are Lipschitz maps defined on $T_\delta$ and  $\varphi=u\circ \mathcal T_\delta^{-1}(\cdot,r)$ and $\psi=v\circ \mathcal T_\delta^{-1}(\cdot,s)$ then  $H_{\varphi,\psi,h}$ interpolates in $\overline T_h$ between $u\res\Gamma_r $ and $v\res \Gamma_s$. 

Let us estimate the gradient and Jacobian determinant of $H_{\varphi,\psi,h}$ in $ T_h^+$: recalling that $\mathcal T_h$ is bi-Lipschitz with constant depending only on $\gamma$, since $\nabla H_{\varphi,\psi,h}(x)=\nabla \Phi_{\varphi,\psi}(\mathcal T_h^{-1}(x))\nabla \mathcal T_h^{-1}(x),$  for a.e. $x\in \overline T_h^+$, one has 
\begin{align*}
&|\nabla H_{\varphi,\psi,h}(x)|\leq |\nabla \Phi_{\varphi,\psi}(\mathcal T_h^{-1}(x))||\nabla \mathcal T_h^{-1}(x)|\leq C_\gamma|\nabla \Phi_{\varphi,\psi}(\mathcal T_h^{-1}(x))|,\\
&\det\big(\nabla H_{\varphi,\psi,h}(x)\big)=\det\big(\nabla \Phi_{\varphi,\psi}(\mathcal T_h^{-1}(x))\big)\det(\nabla \mathcal T_h^{-1}(x)).
\end{align*}
Once again, here and below we denote by $C_\gamma>0$ a constant depending on $\gamma$, but independent of $\delta$, $\varphi$, and $\psi$.
Hence, setting $D=[a,b]\times [0,h]$, one has
\begin{align}
	\int_{T_h^+}|\nabla H_{\varphi,\psi,h}(x)|dx&\leq C_\gamma\int_{T_h^+}|\nabla \Phi_{\varphi,\psi}(\mathcal T_h^{-1}(x))|dx=C_\gamma\int_{D}|\nabla \Phi_{\varphi,\psi}(t,r)||\det(\nabla \mathcal T_h(t,r))|dtdr\nonumber\\
	&\leq C_\gamma\int_{D}|\nabla \Phi_{\varphi,\psi}(t,r)|dtdr,
\end{align}
and analogously on $\overline T_h^-$ 
\begin{align}
	\int_{T_h^-}|\nabla H_{\varphi,\psi,h}(x)|dx&\leq C_\gamma\int_{T_h^-}|\nabla \Psi_{\varphi,\psi}(\widehat {\mathcal T}_h^{-1}(x))|dx\leq C_\gamma\int_{D}|\nabla \Psi_{\varphi,\psi}(t,r)|dtdr.
\end{align}
Therefore, exployting \eqref{stimanablaPhi} and \eqref{stimanablaPsi} we conclude 
\begin{align}\label{est:nablaH}
	\int_{T_h}|\nabla H_{\varphi,\psi,h}(x)|dx&\leq 2C_\gamma h(L_\varphi+b-a)+C_\gamma (L_\varphi+b-a)\int_a^b|s_\psi(t)-s_\varphi(t)|dt\nonumber\\&\;\;\;+C_\gamma(L_\varphi+L_\psi+b-a)h+C_\gamma\int_0^{L_\psi}|{\overline \psi}(s)-{\overline \varphi}(s)|ds\nonumber\\
	&\leq C_{\gamma,L_\varphi,L_\psi}(h+\|s_\psi-s_\varphi\|_{L^1}+\|{\overline \psi}-{\overline \varphi}\|_{L^\infty}),
\end{align}
where the constant $C_{\gamma,L_\varphi,L_\psi}$ is independent of $\delta$, depends on $\gamma$, but is uniformly bounded by a constant $C_\gamma$ (depending only on $\gamma$) as soon as 
$$L_\varphi+L_\psi\leq C,$$
for an absolute constant $C>0$ (notice that $b-a$ coincides with the length of $\Gamma$ and hence we include the dependence on $b-a$ in $C_\gamma$).
Regarding the Jacobian determinant, using \eqref{stimadetPhi} and \eqref{stimadetPsi}, 
we find out that  
\begin{align}\label{est:detH}
&\int_{T_h^+}|\det(\nabla H_{\varphi,\psi,h})(x)|dx=0,\\
&\int_{T_h^-}|\det(\nabla H_{\varphi,\psi,h})(x)|dx\leq \int_D|\det\big(\nabla \Psi_{\varphi,\psi}(t,r)\big)||\det(\nabla \mathcal T_h^{-1}(\mathcal T_h(t,r)))||\det(\nabla \mathcal T_h(t,r))|dtdr\nonumber\\
&\qquad \qquad \qquad \qquad \qquad \qquad\leq (L_\varphi+L_\psi+b-a)\int_0^{1}|{\overline \psi}(s)-{\overline \varphi}(s)|ds\leq C_{\gamma,L_\varphi,L_\psi}\|{\overline \psi}-{\overline \varphi}\|_{L^\infty}.\nonumber
\end{align}
%
%
%
%
%

\textbf{Estimates for the gradient and Jacobian of composition of maps:}
Let $A\subset\R^2$ be an open set and let $B\subset\subset A$ satisfy (R) and be simply-connected. Let $\gamma\in C^3([a,b];\R^2)$ an arc-length parametrization of $\Gamma:=\partial B$. If $B$ is not simply-connected, we will apply the following discussion to each loop forming $\partial B$. Let $\delta\in(0,1)$ be small enough and let $T_\delta$ be a tubular neighborhood of $\Gamma$. For a map $v\in \text{{\rm Lip}}(T_\delta;\R^2)$ we consider the map $$u:=v\circ \Sigma_{\delta,n,c}$$ whose gradient and Jacobian determinant satisfy
\begin{align}\label{composedmap}
	&\nabla u(x)= \nabla v(\Sigma_{\delta,n,c}(x))\nabla\Sigma_{\delta,n,c}(x)=\nabla v(\Sigma_{\delta,n,c}(x))+\nabla v(\Sigma_{\delta,n,c}(x))\sigma_{\delta,n,c}(x)\nonumber\\
	&\det(\nabla u(x))=\det(\nabla v(\Sigma_{\delta,n,c}(x)))+\det(\nabla v(\Sigma_{\delta,n,c}(x)))d_{\delta,n,c}(x),
\end{align}
for a.e. $x\in T_\delta$,
where we have used \eqref{nablaSigma2} and \eqref{detSigma}.
In particular we deduce
\begin{align}
	\int_{T_\delta}|\nabla u(x)-\nabla v(x)|dx&\leq\int_{T_\delta}|\nabla v(\Sigma_{\delta,n,c}(x))-\nabla v(x)|dx +\frac {C_\gamma}{n}\int_{T_\delta}|\nabla v(\Sigma_{\delta,n,c}(x))|dx\nonumber\\
	&\leq \beta_v(\frac1n)+\frac {C_\gamma}{n}(1+ \frac {C_\gamma}{n})\int_{T_\delta} |\nabla v|dx
\end{align}
where in the last inequality we have used \eqref{gradSigma}, \eqref{detSigma}, and where $\beta_v(\frac1n):=\int_{T_\delta}|\nabla v(\Sigma_{\delta,n,c}(x))-\nabla v(x)|dx$. Arguing similarly, we can also estimate
\begin{align}
	\int_{T_\delta}|\det(\nabla u)-\det(\nabla v)|dx&\leq 	\int_{T_\delta}|\det(\nabla v(\Sigma_{\delta,n,c}(x)))-\det(\nabla v(x))|dx\nonumber\\
	&\;\;\;\;+\frac {C_\gamma}{n}(1+ \frac {C_\gamma}{n})\int_{T_\delta}|\det(\nabla v)|dx\nonumber\\
	&\leq \eta_v(\frac1n)+\frac {C_\gamma}{n}(1+ \frac {C_\gamma}{n})\int_{T_\delta}|\det(\nabla v)|dx,
\end{align}
where $\eta_v(\frac1n):=\int_{T_\delta}|\det(\nabla v(\Sigma_{\delta,n,c}(x)))-\det(\nabla v(x))|dx$.
Notice that both the quantities $\beta_v(\frac1n)$ and $\eta_v(\frac1n)$ tend to $0$ as $n\rightarrow \infty$, thanks to the fact that $\Sigma_{\delta,n,c}(x)\rightarrow x$ uniformly.

%
%
%

%
%

Analogously, if we define $u^-:\overline T_\delta^-\setminus T_{\frac cn}\rightarrow \R^2$ and $u^+:\overline T_\delta^+\setminus T_{\frac cn}\rightarrow \R^2$ as $$  u^\pm:=v\circ \Sigma^\pm_{\delta,n,c}$$
respectively, 
then we will have 
\begin{align}\label{grad_inTdelta}
	&\int_{T^-_\delta\setminus T_{\frac cn}}|\nabla u^--\nabla v|dx\leq  \int_{T^-_\delta\setminus T_{\frac cn}}|\nabla v(\Sigma_{\delta,n,c}(x))-\nabla v(x)|dx +\frac {C_\gamma}{n}\int_{T^-_\delta\setminus T_{\frac cn}}|\nabla v(\Sigma_{\delta,n,c}(x))|dx\nonumber\\
	&\qquad\qquad\qquad\qquad\qquad \leq  \beta^-_v(\frac1n)+\frac {C_\gamma}{n}(1+ \frac {C_\gamma}{n})\int_{T_\delta^-} |\nabla v|dx\nonumber\\
	&\int_{T^+_\delta\setminus T_{\frac cn}}|\nabla u^+-\nabla v|dx\leq \beta^+_v(\frac1n)+\frac {C_\gamma}{n}(1+ \frac {C_\gamma}{n})\int_{T_\delta^+} |\nabla v|dx,
\end{align}
and
\begin{align}\label{det_inTdelta}
	\int_{T^-_\delta\setminus T_{\frac cn}}|\det(\nabla u^-)-\det(\nabla v)|dx\leq \eta^-_v(\frac1n)+\frac {C_\gamma}{n}(1+ \frac {C_\gamma}{n})\int_{T_\delta^-}|\det(\nabla v)|dx,\nonumber\\
	\int_{T^+_\delta\setminus T_{\frac cn}}|\det(\nabla u^+)-\det(\nabla v)|dx\leq \eta^+_v(\frac1n)+\frac {C_\gamma}{n}(1+ \frac {C_\gamma}{n})\int_{T_\delta^+}|\det(\nabla v)|dx.
\end{align}
Also in this case the quantities $\beta^\pm_v(\frac1n)$ and $\eta^\pm_v(\frac1n)$ tend to zero  as $n\rightarrow \infty$.

%
%
%

%
%

%
%
%
\section{Main properties of recovery sequences for $\mathcal F$}\label{subsec_properties}
Let $\Gamma:=\gamma([a,b])$, $\gamma\in C^3([a,b];\R^2)$, be a Jordan curve parametrized by arc-length and let $T_\delta$ be a tubular neighborhood of it, for $\delta\in (0,1)$ small enough. 

\begin{definition}[The function $\psi_u$]
If $u\in BV(T_\delta;\R^m)$ we
define the function $\psi_u:(-\delta,\delta)\rightarrow \mathbb R$ as
\begin{align}\label{psi}
	\psi_u(r)=|u\res \Gamma_r|_{BV(\Gamma_r)}=|D_\zeta u|(\Gamma_r),
\end{align}
for all $r\in (-\delta,\delta)$ and where $D_\zeta$ is the tangential distributional derivative of $u$ to $\Gamma_r$ (see Definition \eqref{tang_der}).
\end{definition} 
The function $\psi_u$ turns out to be measurable and, since $u\in BV(T_\delta;\R^m)$, by Coarea formula it belongs to  $L^1((-\delta,\delta))$ (see Lemma \ref{lemma_coareaBV}).

The following result is a crucial lemma which has the role of estimating the errors of energy when one wants to glue two Lipschitz maps along a Jordan curve.

\begin{lemma}\label{lem_glueing1}
	Let $A\subset\R^2$ be a bounded  open set and let $B\subset A$ be a open subset whose boundary is $\partial B=:\Gamma\subset A$ is a  closed Jordan curve of class $C^3$; let $ u\in BV(A;\R^m)$ be such that $$|Du|(\partial B)=0,\qquad u\res\partial B\in BV(\partial B;\R^m),$$ let $v^+,v^-\in\text{{\rm Lip}}_{\text{{\rm loc}}}(A;\R^m)$ be two maps and let $\delta>0$ small so that $T_\delta$ is a tubular neighborhood of $\Gamma$. Then there exists a function $\omega_\Gamma:\R^+\rightarrow \R^+$ depending on $\Gamma$ and on $u\res \Gamma$ (but independent of $v^\pm$) with $\lim_{t\rightarrow 0^+}\omega_\Gamma(t)=0$ and the such that following holds: for all $\varepsilon>0$ there exists a function $w\in\text{{\rm Lip}}_{\text{{\rm loc}}}(A;\R^m)$ with 
	\begin{align}
	&w=v^-\text{ in }B\setminus T_\delta\text{ and }w=v^+\text{ in }A\setminus B\setminus  T_\delta,\nonumber\\
	&\|w-v^-\|_{L^1(T_\delta\cap B)}\leq 3\|v^--u\|_{L^1(T_\delta\cap B)}+r,\nonumber\\
		&\|w-v^+\|_{L^1(T_\delta\cap (A\setminus B))}\leq 3\|v^+-u\|_{L^1(T_\delta\cap (A\setminus B))}+r,\nonumber\\
		&\int_B|\nabla w-\nabla v^-|dx\leq r,\label{glueingthesis}\qquad \qquad \int_{A\setminus B}|\nabla w-\nabla v^+|dx\leq  r,\\
		&F(w, B)\leq F(v^-, B)+r,\;\; F(w,A\setminus \overline B)\leq F(v^+,A\setminus \overline B)+r,\nonumber\\
		&r\leq\varepsilon+\omega_\Gamma\big(d_s(v^+\res \Gamma,u\res\Gamma)+d_s(v^-\res \Gamma,u\res\Gamma)\big).\nonumber
	\end{align}
Moreover, if $v^+,v^-\in\text{{\rm Lip}}(A;\R^m)$ then $w\in\text{{\rm Lip}}(A;\R^m)$.
\end{lemma}
\begin{proof}
	Assume  that  $\gamma:[a,b]\rightarrow \R^2$ is an arc-length parametrization of the loop $\Gamma$. Let us consider the corresponding map $\mathcal T_\delta$ in \eqref{mathcalTdelta} and let $T^-_\delta$ and $T^+_\delta$ denote the interior and external parts of $T_\delta$ with respect to $B$, i.e., $ T^-_\delta=\overline B\cap T_\delta$, $T^+_\delta=T_\delta\setminus B$. Then we set, for any $n\geq 1$,
	\begin{align}
		w_{n}:=\begin{cases}
			v^-\circ \Sigma_{\delta,n,c}^-&\text{ in }T^-_\delta\setminus T_{\frac cn},\\
			v^+\circ  \Sigma_{\delta,n,c}^+&\text{ in }T^+_\delta\setminus T_{\frac cn},\\
			v^-&\text{ in }B\setminus T_\delta,\\
			v^+&\text{ in }A\setminus B\setminus T_\delta,
		\end{cases}
	\end{align}
	where we recall the maps $\Sigma_{\delta,n,c}^-$ and $\Sigma_{\delta,n,c}^+$ in \eqref{Sigma-} and \eqref{Sigma+}, with $c\in (0,\delta)$ fixed.
	We have to define $w_{n}$ in $T_{\frac cn}$: we set 
	$$\widetilde\varphi:=(v^-\circ \Sigma_{\delta,n,c}^- )\res \Gamma_{-\frac cn},\qquad \widetilde\psi:=(v^+\circ \Sigma_{\delta,n,c}^+)\res \Gamma_{\frac cn},$$
	and 
	recalling \eqref{Phi_interp}, \eqref{Psi_interp}, and \eqref{Hinterp}, we define
	\begin{align}
		w_{n}:=H_{\varphi,\psi,\frac cn}\qquad \qquad \text{ in }T_{\frac cn},
	\end{align}
	where $\varphi,\psi:[a,b]\rightarrow \R^m$ are given by
	\begin{align}
		\varphi=\widetilde \varphi\circ \mathcal T_\delta (\cdot,-\frac cn) \qquad \qquad \psi=\widetilde \psi\circ \mathcal T_\delta (\cdot,\frac cn).
	\end{align}
	By definition of $\widetilde\varphi$ and $\widetilde\psi$, using \eqref{Sigma-} and \eqref{Sigma+}, $\varphi$ and $\psi$  can be equivalently written as 
	\begin{align*}
		&\varphi(t)=v^-\circ\mathcal T_\delta\circ \Upsilon_{\delta,n,c}(t,-\frac cn)=v^-(\mathcal T_\delta(t,0))=v^-(\gamma(t))\\
		&\psi(t)=v^+\circ\mathcal T_\delta\circ \Upsilon_{\delta,n,c}(t,\frac cn)=v^+(\mathcal T_\delta(t,0))=v^+(\gamma(t)).
	\end{align*}
%
	In this way we have that $w_{n}$ is Lipschitz continuous in $T_{\frac cn}$ and 
	$$w_{n}=\widetilde \varphi \text{ on }\Gamma_{-\frac cn},\qquad \qquad w_{n}=\widetilde \psi \text{ on }\Gamma_{\frac cn}.$$
	Moreover $w_n$ turns out to be globally  Lipschitz in $A$ if so are $v^+$ and $v^-$.
	Let us estimate the gradient and Jacobian determinant integral of $w_{n}$ in $T_{\delta}$: by \eqref{grad_inTdelta}  we have
	\begin{align}\label{gradient1}
	\int_{T^-_\delta\setminus T_{\frac cn}}|\nabla w_{n}-\nabla v^-|dx&\leq   \beta^-_{v^-}(\frac1n)+\frac {C_\gamma}{n}(1+ \frac {C_\gamma}{n})\int_{T_\delta^-} |\nabla v^-|dx,\nonumber\\
	\int_{T^+_\delta\setminus T_{\frac cn}}|\nabla w_{n}-\nabla v^+|dx&\leq   \beta^+_{v^+}(\frac1n)+\frac {C_\gamma}{n}(1+ \frac {C_\gamma}{n})\int_{T_\delta^+} |\nabla v^+|dx,
	\end{align}
and in particular there is a constant $C_\gamma>0$ (depending on $\gamma$, but independent of $n$) such that 
	\begin{align}
	&\int_{T_\delta\setminus T_{\frac cn}}|\nabla w_{n}|dx\leq \beta^-_{v^-}(\frac1n)+  \beta^+_{v^+}(\frac1n)+ \frac { C_\gamma}{n}\Big(\int_{T^-_\delta} |\nabla v^-|dx+\int_{T^+_\delta} |\nabla v^+|dx\Big),\label{257}
\end{align}
Furthermore, 
on account of \eqref{det_inTdelta}, it follows, for all $i,j\in \{1,\dots,m\},$ $i\neq j$,
\begin{align}
&\int_{T_\delta^-\setminus T_{\frac cn}}|M_{12}^{ij}(\nabla w_{n})-M_{12}^{ij}(\nabla v^-)|dx+\int_{T_\delta^-\setminus T_{\frac cn}}|M_{12}^{ij}(\nabla w_{n})-M_{12}^{ij}(\nabla v^-)|dx\nonumber\\
&\leq  \eta^-_{v^-}(\frac1n)+ \eta^+_{v^+}(\frac1n)+{C_\gamma}{n}\Big(\int_{T^-_\delta}|M_{12}^{ij}(\nabla v^-)|dx+\int_{T^+_\delta}|M_{12}^{ij}(\nabla v^+)|dx\Big).\label{258}
\end{align}
	As for the integral on  $T_{\frac cn}$, using \eqref{est:nablaH} and \eqref{est:detH},  we have for all $i,j\in \{1,\dots,m\},$ $i\neq j$,
	\begin{align}\label{261}
		\int_{T_{\frac{c}{n}}}|\nabla w_{n}|dx&\leq C_{\gamma,L_\varphi,L_\psi}(\frac{1}{n}+\|s_{\psi}-s_{\varphi}\|_{L^1}+\|{\overline \psi}-{\overline \varphi}\|_{L^\infty})\nonumber\\
		&\leq C_{\gamma,L_\varphi,L_\psi}\big(\frac1n+\|s_{\psi}-s_{\sigma}\|_{L^1}+\|s_{\varphi}-s_\sigma\|_{L^1}+\|{\overline \psi}-{\overline\sigma}\|_{L^\infty}+\|{\overline \varphi}-{\overline\sigma}\|_{L^\infty}\big)\nonumber\\
		\int_{T_{\frac{c}{n}}}|M_{12}^{ij}(\nabla w_{n})|dx&\leq C_{\gamma,L_\varphi,L_\psi}\|{\overline \psi}-{\overline \varphi}\|_{L^\infty}\leq C_{\gamma,L_\varphi,L_\psi}\big(\|{\overline \psi}-{\overline \sigma}\|_{L^\infty}+\|{\overline \varphi}-{\overline \sigma}\|_{L^\infty}\big).
	\end{align}
	Here we have set $\sigma:=u\circ \gamma$ and denoted by $\overline \sigma$ the generalized curve in \eqref{gencurve}. By Proposition \ref{prop_convstrict} we find a function $a_\gamma$ such that, up to enlarging the constant $C_{\gamma,L_\varphi,L_\psi}$ if necessary, 
	\begin{align}\label{265}
		\int_{T_{\frac{c}{n}}}|M_{12}^{ij}(\nabla w_{n})|dx+\int_{T_{\frac{c}{n}}}|\nabla w_{n}|dx\leq C_{\gamma,L_\varphi,L_\psi}\big(\frac{1}{n}+a_\gamma(d_s(\varphi,\sigma)+d_s(\psi,\sigma))\big).
	\end{align}
We observe that inequalities 
	\eqref{gradient1}, \eqref{258}, and \eqref{265} entail
	\begin{align*}
		\int_{T^-_\delta}|\mathcal M(\nabla v^-)-\mathcal M(\nabla w_n)|dx+\int_{T^+_\delta}|\mathcal M(\nabla v^+)-\mathcal M(\nabla w_n)|dx
		\leq o(n)+C_\gamma a_\gamma(d_s(\varphi,\sigma)+d_s(\psi,\sigma))
	\end{align*}
	for some quantity $o(n)$ tending to $0$ as $n\rightarrow\infty$. These estimates together with \eqref{growthbis} entail
	\begin{align}\label{268}
		\mathcal F(w_n,B)&=\int_Bg(\mathcal M(\nabla w_n))\leq\int_Bg(\mathcal M(\nabla v^-))-\partial g(\mathcal M(\nabla v^-))(\mathcal M(\nabla v^-)-\mathcal M(\nabla w))dx\nonumber\\
		&\leq \mathcal F(v^-,B)+C_g \int_{T_\delta^-}|\mathcal M(\nabla v^-)-\mathcal M(\nabla w)| dx=:\mathcal F(v^-,B)+r',
	\end{align} 
	with $r'\leq C_g(o(n)+C_\gamma a_\gamma(d_s(\varphi,\sigma)+d_s(\psi,\sigma))) =:C_g o(n)+\omega_\Gamma(d_s(\varphi,\sigma)+d_s(\psi,\sigma))$. A similar reasoning  for the set $A\setminus B$ leads to
	\begin{align*}
	F(w,A\setminus B)=F(v^+,A\setminus B)+r'',
	\end{align*}
with $r''\leq C_g o(n)+\omega_\Gamma(d_s(\varphi,\sigma)+d_s(\psi,\sigma))$. So if we take $n$ large enough, we have obtained the last but one line in \eqref{glueingthesis}.
Also  the forth inequality in \eqref{glueingthesis} easily follows from \eqref{gradient1} and \eqref{261}.
It remains to estimate the $L^1$-norms. Owing to the explicit expression of $\Phi_{\varphi,\psi}$ and $H_{\varphi,\psi,\frac cn}$ in \eqref{Phi_interp} and  \eqref{Hinterp}, denoting $h=\frac cn$, we write
\begin{align*}
\int_{T^-_{\frac{c}{n}}}|w_n|dx&\leq (1+\frac{C_\gamma}{n})\int_a^b\int_0^h|\varphi\Big(t_\varphi\big(s_\varphi(t)\frac{r}{h}+s_\psi(t)\frac{h-r}{h} \big)\Big)|drdt\\
&=(1+\frac{C_\gamma}{n})\int_a^b\int_0^h|\overline \varphi\big(s_\varphi(t)\frac{r}{h}+s_\psi(t)\frac{h-r}{h} \big)|drdt\\
&\leq \frac{c}{n}(b-a)(1+\frac{C_\gamma}{n})\big(\|\overline \varphi-\overline \sigma\|_{L^\infty}+\|\overline \sigma\|_{L^\infty}\big)\\
&\leq \frac{C_\gamma}{n} \big(a_\gamma(d_s(\varphi,\sigma))+\|\sigma\|_{L^\infty}\big),
\end{align*}
where the first inequality follows from \eqref{detT-1} and the last one from Proposition \ref{prop_convstrict}. Analogously
\begin{align*}
	\int_{T^+_{\frac{c}{n}}}|w_n|dx&\leq (1+\frac{C_\gamma}{n})\int_a^b\int_0^h|\overline \varphi(s_\psi(t))\frac{h-r}{h}+\overline \psi(s_\psi(t))\frac{r}{h}|drdt\\&
	\leq \frac{C_\gamma}{n}\big(\|\overline \varphi-\overline \sigma\|_{L^\infty}+\|\overline \psi-\overline \sigma\|_{L^\infty}+2\|\overline \sigma\|_{L^\infty}\big)\\
	&\leq \frac{C_\gamma}{n} \big(a_\gamma(d_s(\varphi,\sigma))+a_\gamma(d_s(\psi,\sigma))+\|\sigma\|_{L^\infty}\big).
\end{align*}
At the same time we have
\begin{align*}
\int_B&|w_n-v^-|dx=\int_{T_\delta\setminus T_{\frac cn}}|v^-\circ\Sigma_{\delta,n,c}^- -v^-|dx\\
&\leq  \int_{T_\delta\setminus T_{\frac cn}}|v^-\circ\Sigma_{\delta,n,c}^- -u\circ\Sigma_{\delta,n,c}^-|dx+\int_{T_\delta\setminus T_{\frac cn}}|u\circ\Sigma_{\delta,n,c}^--u|dx+\int_{T_\delta\setminus T_{\frac cn}}|u-v^-|dx\\
&\leq (1+\frac{C_\gamma}{n})\|v^--u\|_{L^1(T_\delta^-)}+\int_{T_\delta\setminus T_{\frac cn}}|u\circ\Sigma_{\delta,n,c}^--u|dx+\|v^--u\|_{L^1(T_\delta^-)}\\
&\leq 3 \|v^--u\|_{L^1(T_\delta^-)}+\int_{T_\delta}|u\circ\Sigma_{\delta,n,c}-u|dx,
\end{align*}
(for $n>1/C_\gamma$)
and a similar inequality holds for $\int_{A\setminus B}|w_n-v^+|dx$.
Hence, the second and third inequalities in \eqref{glueingthesis} follow from the last three expressions, noticing that we can choose $n$ big enough so that 
\begin{align*}
\frac {C_\gamma}{n}\|\sigma\|_{L^\infty}\leq \varepsilon,\qquad \qquad \int_{T_\delta}|u\circ\Sigma_{\delta,n,c}-u|dx\leq \varepsilon,
\end{align*}
where the last condition can be obtained because  $u\circ\Sigma_{\delta,n,c}\rightarrow u$ in $L^1(T_\delta;\R^m)$ as $n\rightarrow \infty$.
\end{proof}

Being the construction leading to the result above local, it can be easily extended to more general open set $B$ as follows:

\begin{cor}\label{cor_glueing1}
	Let $A$ be a bounded open set and let $B\subset A$ be an open subset with boundary  $\partial B\subset A$ a finite union of  closed curves of class $C^3$; let $u\in BV(A;\R^m)$ be such that$$|Du|(\partial B)=0,\qquad u\res\partial B\in BV(\partial B;\R^m),$$  let $v^+,v^-\in\text{{\rm Lip}}_{\text{{\rm loc}}}(A;\R^m)$ be two maps and let $\delta>0$ small so that $T_\delta$ is a tubular neighborhood of $\Gamma:=\partial B$. Then for all $\varepsilon>0$ there exists $w\in\text{{\rm Lip}}_{\text{{\rm loc}}}(A;\R^m)$  such that the first five lines of \eqref{glueingthesis} hold, together with
	\begin{align}
	r\leq\varepsilon+\sum_{i=0}^N\omega_{\Gamma^i}\big(d_s(v^+\res \Gamma^i,u\res\Gamma^i)+d_s(v^-\res \Gamma^i,u\res\Gamma^i)\big),\label{269}
	\end{align}
where $\omega_{\Gamma^i}:\R^+\rightarrow \R^+$ are functions depending on $\Gamma^i$ and on $u\res \Gamma^i$ respectively, such that $\lim_{t\rightarrow 0^+}\omega_{\Gamma^i}(t)=0$. Also, if $v^+,v^-\in\text{{\rm Lip}}(A;\R^m)$ then $w\in \text{{\rm Lip}}(A;\R^m)$.
\end{cor}

A straightforward consequence of the previous result is the following:

\begin{cor}\label{cor_glueing}
	Let $A$ be a bounded open set and let $B\subset A$ be an open subset with boundary  $\partial B$ a finite union of  closed curves of class $C^3$; let $u\in BV(A;\R^m)$ be such that $$|Du|(\partial B)=0,\qquad u\res\partial B\in BV(\partial B;\R^m),$$  and let $(u_k),(v_k)\subset \text{{\rm Lip}}_{\text{{\rm loc}}}(A;\R^m)$ be two sequences of maps such that 
	\begin{align*}
		&v_k\rightarrow u\;\text{ and }\;u_k\rightarrow u \qquad \text{strictly in }BV(A;\R^m),\\
		&u_k\res\Gamma^i\rightarrow u\res \Gamma^i\;\text{ and }\;v_k\res \Gamma^i\rightarrow  u\res \Gamma^i \qquad \text{strictly in }BV(\Gamma^i;\R^m),
	\end{align*}
	where $\Gamma=\cup_{i=0}^N\Gamma^i$ is the decomposition of $\Gamma$ in simple Jordan curves $\Gamma^i$.
	Then there exists a sequence $(w_k)\subset \text{{\rm Lip}}_{\text{{\rm loc}}}(A;\R^m)$ such that 
	\begin{align*}
	&w_k\rightarrow u\text{ strictly in }BV(A;\R^m),\\
	&\liminf_{k\rightarrow\infty}F(w_k,B)\leq \liminf_{k\rightarrow\infty}F(v_k,B),\\
	&\liminf_{k\rightarrow\infty}F(w_k,A\setminus \overline B)\leq \liminf_{k\rightarrow\infty}F(u_k,A\setminus \overline B).
	\end{align*} 
\end{cor}

We now use the previous result to modify suitable recovery sequences.

\begin{lemma}\label{lem_recoverymodified}
Let $A$ be a bounded open set and let $B\subset\subset A$ be a open subset whose boundary is $\partial B=:\Gamma$ a finite union of closed curves of class $C^3$. Let $u\in BV(A;\R^m)$ be given and assume that $0$ is a regular value for the function $\psi$ in \eqref{psi}. Then there exists a recovery sequence $(v_k)\subset  \text{{\rm Lip}}(B;\R^m)$ for $\mathcal F(u,B)$ such that $v_k\res\Gamma\rightarrow u\res \Gamma$ strictly in $BV(\Gamma;\R^m)$.
\end{lemma}
\begin{proof}
Let $(u_k)\subset \text{{\rm Lip}}_{\text{{\rm loc}}}(B;\R^m)$ be a recovery sequence for $\mathcal F(u,B)$, let $T_\delta$ be a tubular neighborhood of $\Gamma$, with $\delta\in(0,1)$ small enough. We will modify $u_k$ in $T_\delta^-$ in order to produce $v_k$. To do so, we again assume that $\Gamma$ consists of a unique loop (the same argument applied to each component of $\Gamma$  covers the general case). Let $\Sigma^-_{\delta,n,c_n}$ be the map in \eqref{Sigma-}, where we consider the numbers $c_n\in (0,\delta)$ in such a way that 
\begin{align}\label{limiteregolare}
\lim_{n\rightarrow \infty}\psi(-\frac{c_n}{n})=\psi(0),
\end{align}
and, at the same time, for all $n>0$ fixed
\begin{align*}
u_k\res\Gamma_{-\frac{c_n}{n}}\rightarrow u\res \Gamma_{-\frac{c_n}{n}}\qquad \text{ strictly in } BV(\Gamma_{-\frac{c_n}{n}};\R^m).
\end{align*}
This choice is possible thanks to the hypothesis that $0$ is regular for $\psi$, and since the convergence above holds on $\Gamma_t$, for a.e. $t\in (-\delta,0)$.
Then we define
\begin{align*}
	v_{k,n}(x):=u_k\left((\Sigma^-_{\delta,n,c_n})^{-1}(x)\right),\qquad x\in T_\delta^-.
\end{align*}
Notice that $(\Sigma^-_{\delta,n,c_n})^{-1}:\overline T_\delta^-\rightarrow \overline T_\delta^-\setminus T_{\frac{c_n}{n}}$ is such that $(\Sigma^-_{\delta,n,c_n})^{-1}(\Gamma)=\Gamma_{-\frac{c_n}{n}}$, and so, writing $x=\mathcal T_\delta(t,0)$ for $x\in \Gamma$, $t\in [a,b]$, we have 
\begin{align*}
	v_{k,n}(\mathcal T_\delta(t,0))=u_k(\mathcal T_\delta \circ \Upsilon^{-1}_{\delta,n,c_n}(t,0))=u_k(\mathcal T_\delta(t,\tau_{\delta,n,c_n}^{-1}(0)))=u_k(\mathcal T_\delta(t,-\frac{c_n}{n})),
\end{align*}
for all $t\in [a,b]$. In particular Remark \ref{rem_stricttraces} implies that, for all $n>0$ fixed
$$v_{k,n}\circ\mathcal T_\delta(\cdot,0)\rightarrow u\circ\mathcal T_\delta(\cdot,-\frac{c_n}{n}) \qquad \text{strictly in }BV([a,b];\R^m).$$
We can then find, for all $k>0$, a natural number $n_k>0$ such that $n_k\nearrow+\infty$ (as $k\rightarrow\infty$) and  satisfying
\begin{align*}
	\int_{\Gamma}|\nabla v_{k,n_k}\zeta|d\mathcal H^1=
\int_a^b|\frac{d}{dt}v_{k,n_k}(\mathcal T_\delta(t,0))|dt\leq |D_\zeta u(\mathcal T_\delta(\cdot,-\frac{c_{n_k}}{n_k}))|([a,b])+\frac1k,
\end{align*}
where $\zeta$ appears in \eqref{zeta} which, we recall, is the unit oriented tangent vector to $\Gamma$.
Recalling also \eqref{Dugammar}, we also have
 \begin{align*}
|D_\zeta u(\mathcal T_\delta(\cdot,-\frac{c_{n_k}}{n_k}))|([a,b])=|D_\zeta u|(\Gamma_{-\frac{c_{n_k}}{n_k}})
 \end{align*}
so we readily infer, thanks to \eqref{limiteregolare} and the lower semicontinuity of the variation, that 
\begin{align*}
\lim_{k\rightarrow\infty}	\int_{\Gamma}|\nabla v_{k,n_k}\zeta|d\mathcal H^1=|D_\zeta u|(\Gamma),
\end{align*}
and therefore the function $v_k:=v_{k,n_k}$ satisfies
\begin{align*}
	v_k\res\Gamma\rightarrow u\res \Gamma\qquad \text{ strictly in } BV(\Gamma;\R^m).
\end{align*}
To conclude the proof we need to show that $v_k$ is still a recovery sequence for $\mathcal F(u,B)$. Notice that, since $u_k$ are Lipschitz in $B\setminus T_{\frac{c_n}{n}}$ and $\Sigma^-_{\delta,n,c_n}$ is bi-Lipschitz, also $v_{k,n}$ are Lipschitz continuous on $B$. In $T_\delta^-$, arguing as in \eqref{composedmap}, it holds, for $i,j\in \{1,\dots,m\}$, $i\neq j$,
\begin{align}\label{adesso}
	&\nabla v_k(x)= \nabla u_k(\Sigma_{\delta,n_k,c_{n_k}}^{-1}(x))\nabla\Sigma_{\delta,n_k,c_{n_k}}^{-1}(x)=\nabla u_k(\Sigma_{\delta,n_k,c_{n_k}}(x))+\nabla u_k(\Sigma_{\delta,n_k,c_{n_k}}(x))\widehat \sigma^-_{\delta,n_k,c_{n_k}}(x),\nonumber\\
	&M^{ij}_{12}(\nabla v_k(x))=M^{ij}_{12}(\nabla u_k(\Sigma_{\delta,n_k,c_{n_k}}(x)))+M^{ij}_{12}(\nabla u_k(\Sigma_{\delta,n_k,c_{n_k}}(x)))\widehat d^-_{\delta,n_k,c_{n_k}}(x),
\end{align}
thanks to \eqref{nablaSigmapm}.
We then introduce the vector 
\begin{align}
\widetilde {\mathcal M}(\nabla u_k(\Sigma_{\delta,n_k,c_{n_k}}^{-1}(x))):=\big(1,\nabla u_k(\Sigma_{\delta,n_k,c_{n_k}}^{-1}(x)),M_{12}(\nabla u_k(\Sigma^{-1}_{\delta,n_k,c_{n_k}}(x)))\big).
\end{align}
where to shotcut the notation, we have denoted $M_{12}(\nabla w)\in \R^{m(m+1)/2}$ the vector with entries $M^{ij}_{12}(\nabla w)$, $i,j\in \{1,\dots,m\}$, $i\neq j$.  
Using \eqref{nablaSigmapm} we infer
\begin{align*}
|\mathcal M(\nabla v_k)-\widetilde {\mathcal M}(\nabla u_k(\Sigma_{\delta,n_k,c_{n_k}}^{-1}(x)))|\leq \frac{C_\gamma}{n_k}\Big|\big(0,\nabla u_k(\Sigma_{\delta,n_k,c_{n_k}}^{-1}(x)),M_{12}(\nabla u_k(\Sigma^{-1}_{\delta,n_k,c_{n_k}}(x)))\big)\Big|.
\end{align*} 
Therefore, exployting \eqref{growthbis}, \eqref{detSigma}, and the convexity of $g$, we can estimate
\begin{align*}
	F(v_k,T_\delta^-)&\leq \int_{T_\delta^-}|g(\mathcal M(\nabla v_k))-g\big(\widetilde {\mathcal M}(\nabla u_k(\Sigma_{\delta,n_k,c_{n_k}}^{-1}(x)))\big)|dx+ \int_{T_\delta^-}g\big(\widetilde {\mathcal M}(\nabla u_k(\Sigma_{\delta,n_k,c_{n_k}}^{-1}(x)))\big)dx\\
	&\leq \frac{C_gC_\gamma}{n_k}\int_{T_\delta^-}\Big|\big(0,\nabla u_k(\Sigma_{\delta,n_k,c_{n_k}}^{-1}(x)),M_{12}(\nabla u_k(\Sigma^{-1}_{\delta,n_k,c_{n_k}}(x)))\big)\Big|dx\\
	&\;\;\;\;+\int_{T_\delta^-}g\big(\widetilde {\mathcal M}(\nabla u_k(\Sigma_{\delta,n_k,c_{n_k}}^{-1}(x)))\big)dx\\
	&=\frac{C_gC_\gamma}{n_k}\int_{T_\delta^-\setminus \overline T_{\frac{c_{n_k}}{n_k}}}\Big|\big(0,\nabla u_k(y),M_{12}(\nabla u_k(y)\big)\Big||\det(\nabla \Sigma_{\delta,n_k,c_{n_k}}(y))|dy\\	&\;\;\;\;+\int_{T_\delta^-\setminus \overline T_{\frac{c_{n_k}}{n_k}}}g\big(\widetilde {\mathcal M}(\nabla u_k(y))\big)|\det(\nabla \Sigma_{\delta,n_k,c_{n_k}}(y))|dy\\
	&\leq \frac{C_gC_\gamma}{n_k}(1+\frac{C_\gamma}{n_k})\int_{T_\delta^-\setminus \overline T_{\frac{c_{n_k}}{n_k}}}\Big|\big(0,\nabla u_k(y),M_{12}(\nabla u_k(y)\big)\Big|dy\\	&\;\;\;\;+(1+\frac{C_\gamma}{n_k})\int_{T_\delta^-\setminus \overline T_{\frac{c_{n_k}}{n_k}}}g\big(\widetilde {\mathcal M}(\nabla u_k(y))\big)dy\\
	&\leq \frac{C_gC_\gamma}{n_k}(1+\frac{C_\gamma}{n_k})\left(|\nabla u_k|(A)+|M_{12}(\nabla u_k)|(A)\right)+(1+\frac{C_\gamma}{n_k})F(u_k,T_\delta^-)
\end{align*}
and so, thanks to \eqref{growth2},  we conclude, for $k$ large enough,
\begin{align*}
	F(v_k,T_\delta^-)\leq 	F(u_k,T_\delta^-)+ \frac{C_{\gamma,g}}{n_k}\left(|\nabla u_k|(A)+F(u_k,A)\right),
\end{align*}
for a constant $C_{\gamma,g}>0$ depending on $\gamma$, $g$, but independent on $u_k$ and $k$.
%
As a consequence, using that $u_k$ is a recovery sequence and that it is  converging to $u$ strictly in $BV(\Om;\R^m)$, we are led to 
\begin{align*}
\limsup_{k\rightarrow \infty} 	F(v_k,T_\delta^-)\leq \lim_{k\rightarrow \infty} 	F(u_k,T_\delta^-),
\end{align*}
which means that $v_k$ is a recovery sequence as well, thanks to the fact that $v_k$ still converges to $u$ strictly in $BV(\Om;\R^m)$ (how it is easily checked from \eqref{adesso}).
\end{proof}

\begin{prop}\label{prop:recoveryrestriction}
Let $A$ be a bounded open set and let $B\subset\subset A$ be a open subset whose boundary is $\partial B=:\Gamma\subset A$ a finite union of  closed curves of class $C^3$. Let $T_\delta\subset A$ be a tubular neighborhood of $\Gamma$, let  $\psi:(-\delta,\delta)\rightarrow \R$ be the function defined in \eqref{psi}, and assume that $0$ is a regular value for $\psi$. Let $(u_k)\subset \text{{\rm Lip}}_{\text{{\rm loc}}}(A;\R^m)$ be a recovery sequence for $\mathcal A(u;A)$ such that $u_k\res\Gamma\rightarrow u\res \Gamma$ strictly in $BV(\Gamma;\R^m)$; then $u_k\res B$ is a recovery sequence for $\mathcal A(u;B)$.
\end{prop}
\begin{proof} 
We prove the assertion arguing by contradiction, so assume that $u_k$ is not a recovery sequence for $\mathcal F(u,B)$; we can then extract a subsequence such that there exists the limit 
$$\lim_{k\rightarrow \infty}F(u_k,B)>\mathcal F(u,B).$$
Let  $(v_k)\subset  \text{{\rm Lip}}_{\text{{\rm loc}}}(B;\R^m)$ be a recovery sequence for $\mathcal F(u,B)$ so that 
$$\mathcal F(u,B)=\lim_{k\rightarrow\infty} F(v_k,B)<\lim_{k\rightarrow\infty} F(u_k,B).$$ 
According to Lemma \ref{lem_recoverymodified}, we can suppose that $v_k\res \Gamma \rightarrow u\res \Gamma$ strictly in $BV(\Gamma;\R^m)$, and that $v_k$ are Lipschitz continuous on $B$. Therefore, the same being true for $u_k\res \Gamma$, we are in the hypotheses of Corollary \ref{cor_glueing}, and we can find a sequence $w_k\in \text{{\rm Lip}}_{\text{{\rm loc}}}(A;\R^m)$ such that
\begin{align*}
\lim_{k\rightarrow\infty} F(w_k,A)&=\lim_{k\rightarrow\infty} F(w_k,A\setminus \overline B)+\lim_{k\rightarrow\infty} F(w_k,B)=\lim_{k\rightarrow\infty} F(u_k,A\setminus \overline B)+\lim_{k\rightarrow\infty} F(v_k,B)\nonumber\\
&<\lim_{k\rightarrow\infty} F(u_k,A\setminus \overline B)+\lim_{k\rightarrow\infty} F(u_k,B)=\lim_{k\rightarrow\infty} F(u_k,A)=\mathcal F(u,A),
\end{align*}
that is absurd. The thesis follows.
\end{proof}
\section{Proof of Theorem \ref{teo_main1}: Monotonicity, inner regularity and sub-addivitity}\label{sec_proof}
This Section is devoted to the proof of Theorem \ref{teo_main1}. To this purpose we need to use Theorem \ref{DGL_teo}, and so we will check that hypohteses (i)-(iv) of that theorem are satisfied. We start with the following technical result: 
\begin{prop}\label{prop31}
	Let $A\subset\Om$ be open and let $(u_k)\subset \text{{\rm Lip}}_{\text{{\rm loc}}}(\Om;\R^m)$ be a sequence such that $u_k\rightarrow u$ strictly in $BV(\Om;\R^m)$; then there exists a sequence $(w_j)\subset \text{{\rm Lip}}_{\text{{\rm loc}}}(A;\R^m)$ such that the following holds:
	\begin{itemize}
\item[(i)] $w_j\rightarrow u$ strictly in $BV(A;\R^m)$;
\item[(ii)] $\liminf_{j\rightarrow \infty}F(w_j,A)\leq \liminf_{k\rightarrow \infty}F(u_k,A).$
	\end{itemize}

\end{prop}
\begin{proof}

\textit{Step 1:} (Setup)
As $A\subset\R^2$ is bounded, we consider the set $\Sigma_n\subseteq A$ defined by 
$$\Sigma_n:=\{x\in A:\text{{\rm dist}}(x,A^c)=\eta_n\},$$
where the numbers $\eta_n$,  are chosen so that for all $n\geq1$ it holds
$0<\eta_{n+1}<\eta_n$, and  $\Sigma_n$ is a finite union of Lipschitz loops $\Sigma_n=\cup_{i=1}^{N_n}\Sigma^i_n$ (see Lemma \ref{lem_Fu} in Appendix). 
We assume that $\Sigma_n^i$ is a unique Jordan curve for all $i=1,\dots,N_n$. 
Let \begin{align}\label{dn}
	d_n:=\min\{\text{{\rm dist}}(\Sigma_n^i,\Sigma_n^j),\;0\leq i<j\leq N_n\},
\end{align}
and for all $i=1,\dots,N_n$ we choose a simple loop $\widehat \Gamma_n^i$ of class $C^4$ such that  
$$\widehat \Gamma_n^i\subset\{x\in A:\dist(x,A^c)\in (\eta_{n+1},\eta_n),\; \dist(x,\Sigma_n^i)<\frac{d_n}{4}\},$$
and in such a way that the region enclosed by $ \widehat \Gamma_n^i$ and $\Sigma_n^i$ is an annulus type open set contained in $\{x\in A:\dist(x,A^c)\in (\eta_{n+1},\eta_n)\}$. For all $i$, we denote by $\widehat H_{n}^i$ this annulus so that
$$\widehat H_{n}^i\subset \{x\in A:\dist(x,A^c)\in (\eta_{n+1},\eta_n)\},\qquad \partial \widehat H_{n}^i= \widehat \Gamma_n^i\cup  \Sigma_n^i.$$
Furthermore we consider tubular neighborhoods $T_{\widehat \delta_n^i}$ of $\widehat \Gamma_n^i$ with $\widehat \delta_{n}^i>0$ so small in order that $$T_{\widehat \delta_n^i}\subset \{x\in A:\dist(x,A^c)\in (\eta_{n+1},\eta_n),\; \dist(x, \Sigma_n^i)<\frac{d_n}{2}\}.$$
Notice that, thanks to our choice of the parameters, it turns out that the open sets $T_{\widehat \delta_{n}^i}$, $n\in\mathbb N$, $i=1,\dots,N_n$, are mutually disjoint.

Let now $(u_k)\subset \text{{\rm Lip}}_{\text{{\rm loc}}}(\Om;\R^m)$ be a sequence as in the statement. For all $n\geq1$ and all $i=1,\dots,N_n$ we choose a positive number $r_n^i<\widehat \delta_n^i/2$ such that, setting, as usual, 
$$(\widehat \Gamma_n^i)_r:=\{x\in T_{\widehat \delta^n_i}:\dist(x, \widehat \Gamma_n^i)=r\}$$
the following conditions hold:
\begin{itemize}
	\item[(a)] $|Du|((\widehat \Gamma_n^i)_{r_n^i})=0$ and $u\res (\widehat \Gamma_n^i)_{r_n^i}$ belongs to $BV((\Gamma_n^i)_{r_n^i};\R^m)$;
	\item[(b)] Setting $\widehat \psi_n^i(r):=|u\res (\widehat \Gamma_n^i)_r|_{BV}=|D_\zeta u|((\widehat \Gamma_n^i)_r)$ then $r_n^i$ is a regular value for $\widehat \psi_n^i$;
	\item[(c)]  $u_k\res (\widehat \Gamma_n^i)_{r_n^i}\rightarrow u\res (\widehat \Gamma_n^i)_{r_n^i}$ strictly in $BV((\widehat \Gamma_n^i)_{r_n^i};\R^m)$.
\end{itemize}
We notice that the loops $(\widehat \Gamma_n^i)_{r^i_n}$ are of class $C^3$ and we denote $$\Gamma_n^i:=(\widehat \Gamma_n^i)_{r^i_n};$$ 
let $H^i_n$ be the annulus type region enclosed by $\Sigma_n^i$ and $\Gamma_n^i$, so that $H_n^i\subset\widehat H_n^i$. In this way conditions (a), (b), and (c), are satisfied for $\Gamma_n^i$ replacing $(\widehat \Gamma_n^i)_{r^i_n}$ and $0$ is a regular value for $\psi_n^i(r):=|u\res (\Gamma_n^i)_r|_{BV}$; finally, since $r_{n}^i<\widehat \delta_n^i/2$ the tubular neighborhoods $T_{\delta_n^i}$ of $\Gamma_{n}^i$, with $\delta_n^i:=\widehat \delta_n^i/2$, for $n>0$, $i=1,\dots,N_n$ are all mutually disjoint.

For any integer $n>0$ fixed, we consider the open set $B_n$ defined as $$B_n:=\overline A_n\cup \bigcup_{i=1}^{N_n}H_n^i.$$
In this way and by definition of $H_n^i$, we see that for all $n>1$ it holds 
$$B_n\subset\subset A,\qquad \qquad A_{n}\subset B_n\subset A_{n+1}.$$
\textit{Step 2:} We now fix a natural number $j>0$ and for all $n\geq 1$ we consider the functions $\omega_{\Gamma^i}:=\omega_{\Gamma_n^i}$ appearing in the right-hand side of \eqref{269}; then we choose a number $a_{n}>0$ so that 
\begin{align}\label{32}
\sum_{i=1}^{N_n}\omega_{\Gamma^i}(t)<\frac{1}{j2^{n+1}}\qquad \text{ for all }t<a_n.
\end{align}
For $n=1$ we consider the set $B_1$ and owing to conditions (a), (b), and (c), we choose a natural number $k_{1,j}>0$ so that 
\begin{itemize}
\item[(1)] $d_s(u_{k_{1,j}}\res \Gamma_1^i,u\res \Gamma_1^i)<\frac{a_1}{2}$, for all $i=1,\dots, N_1$;
\item[(2)] $\|u_{k_{1,j}}-u\|_{L^1(B_1)}+||Du_{k_{1,j}}|(B_1)-|Du|(B_1)|<\frac{1}{4j}$;
\item[(3)] $F(u_{k_{1,j}},B_1)\leq \liminf_{k\rightarrow \infty}F(u_k,B_1)+\frac{1}{4j}$.
\end{itemize}
Next, for every $n>1$ we  choose $k_{n,j}>k_{n-1,j}$ so that the following holds
\begin{itemize}
	\item[(1*)] $d_s(u_{k_{n,j}}\res \Gamma_n^i,u\res \Gamma_n^i)<\frac{a_n}{2}$, $\forall i=1,\dots, N_n$ and  $d_s(u_{k_{n,j}}\res \Gamma_{n-1}^i,u\res \Gamma_{n-1}^i)<\frac{a_{n-1}}{2}$, $\forall i=1,\dots, N_{n-1}$;
	\item[(2*)] $\|u_{k_{n,j}}-u\|_{L^1(B_n\setminus B_{n-1})}+||Du_{k_{n,j}}|(B_n\setminus B_{n-1})-|Du|(B_n\setminus B_{n-1})|<\frac{1}{j2^{n+1}}$;
	\item[(3*)] $F(u_{k_{n,j}},B_n\setminus \overline B_{n-1})\leq \liminf_{k\rightarrow \infty}F(u_k,B_n\setminus \overline B_{n-1})+\frac{1}{j2^{n+1}}$.
\end{itemize}
Conditions (2) and (2*) can be obtained because $u_k\rightarrow u$ strictly in $BV(\Om;\R^m)$, and thanks to the hypothesis that $|Du|$ does not concentrate on $\partial B_n$, for any $n\geq1$ (so the strict convergence is inerhited on $B_n\setminus \overline B_{n-1}$).

\textit{Step 3:} We now proceed to glue the maps $u_{k_{n,j}}$ along the tubes $T_{\delta_n^i}$ exploiting Lemma \ref{lem_glueing1}.
More precisely, for all $n\geq1$ we apply Corollary \ref{cor_glueing1} with $A,B$ replaced by $B_{n+1}\setminus \overline B_{n-1}$ and $B_{n}\setminus\overline  B_{n-1} $, respectively, $\delta_n=\min\{\delta_n^i,\;{i=1,\dots,N_n}\}$, and $\varepsilon=\frac{1}{j2^{n+1}}$, $v^-=u_{k_{n,j}}$, $v^+=u_{k_{n+1,j}}$. This provides us with a map $w_{n,j}\in \text{{\rm Lip}}(B_{n+1}\setminus \overline B_{n-1};\R^m)$, (here we have set $B_0=\varnothing$ to include the case $n=1$) such that 
	\begin{align}
	&w_{n,j}=u_{k_{n,j}}\text{ in }B_{n}\setminus\overline  B_{n-1} \setminus \overline T_{\delta_n}\text{ and }w_{j,n}=u_{k_{n+1,j}}\text{ in }B_{n+1}\setminus\overline  B_{n} \setminus  \overline T_{\delta_n},\nonumber\\
	&\|w_{n,j}-u_{k_{n,j}}\|_{L^1(T_{\delta_n}\cap B_n)}\leq 3\|u_{k_{n,j}}-u\|_{L^1(T_{\delta_n}\cap B_n)}+r_{n,j},\nonumber\\
	&\|w_{n,j}-u_{k_{n+1,j}}\|_{L^1(T_{\delta_n}\cap (B_{n+1}\setminus\overline B_n))}\leq 3\|u_{k_{n+1,j}}-u\|_{L^1(T_{\delta_n}\cap (B_{n+1}\setminus\overline B_n))}+r_{n,j},\nonumber\\
	&\int_{B_{n}\setminus\overline  B_{n-1}}|\nabla w_{n,j}-\nabla u_{k_{n,j}}|dx\leq r_{n,j},\label{glueingthesis2}\\
	&\int_{B_{n+1}\setminus\overline  B_{n}}|\nabla w_{n,j}-\nabla u_{k_{n+1,j}}|dx\leq r_{n,j},\nonumber\\
	&F(w_{n,j};B_{n}\setminus\overline  B_{n-1})\leq F(u_{k_{n,j}}; B_{n}\setminus\overline  B_{n-1})+r_{n,j},\nonumber\\
	&F(w_{n,j};B_{n+1}\setminus\overline  B_{n})\leq F(u_{k_{n+1,j}};B_{n+1}\setminus\overline  B_{n})+r_{n,j},\nonumber\\
	&r_{n,j}\leq \frac{1}{j2^{n+1}}+\sum_{i=0}^{N_n}\omega_{\Gamma^i}\big(d_s(u_{k_{n,j}}\res \Gamma_n^i,u\res \Gamma_n^i)+d_s(u_{k_{n+1,j}}\res \Gamma_n^i,u\res \Gamma_n^i)\big)\leq  \frac{1}{j2^{n}},\nonumber
\end{align}
where the last inequality is obtained in view of (1*) (and also (1)), thanks to \eqref{32}. Due to the first line, we can now define $w_j\in \text{{\rm Lip}}_{\text{loc}}(A;\R^m)$ as
$$w_j:=w_{n,j}\qquad\text{ on }\qquad U_n:=(B_{n}\setminus\overline  B_{n-1}\setminus \overline T_{\delta_{n-1}})\cup T_{\delta_n}.$$
We can now estimate 
\begin{align*}
\|w_j-u\|_{L^1(A)}&\leq \sum_{n=1}^\infty \|w_{n,j}-u_{k_{n,j}}\|_{L^1(T_{\delta_n}\cap B_n)}+\|w_{n,j}-u_{k_{n+1,j}}\|_{L^1(T_{\delta_n}\cap (B_{n+1}\setminus\overline B_n))}\nonumber\\
&\leq \sum_{n=1}^\infty 2r_{n,j}+3\|u-u_{k_{n,j}}\|_{L^1(T_{\delta_n}\cap B_n)}+3\|u-u_{k_{n+1,j}}\|_{L^1(T_{\delta_n}\cap (B_{n+1}\setminus\overline B_n))}
\end{align*}
where we have used \eqref{glueingthesis2}, and thanks to (2) and (2*) we conclude
\begin{align}\label{tesi1}
	\|w_j-u\|_{L^1(A)}\leq\frac5j.
\end{align}
A similar argument applied to the forth and fifth lines in \eqref{glueingthesis2} and again based on (2) and (2*) leads to
\begin{align}\label{tesi2}
	|Dw_j|(A)&\leq \sum_{n=1}^{\infty}|Du_{k_{n,j}}|(T_{\delta_n}\cap B_n )|+|Du_{k_{n+1,j}}|(T_{\delta_n}\cap B_{n+1} )|+2r_{n,j}\leq |Du|(A)+\frac3j.
\end{align}
Finally, arguing analogously, thanks to the first, sixth, and seventh line in \eqref{glueingthesis2} and to (3) and (3*) we conclude 
\begin{align}\label{tesi3}
F(w_j,A)&=\sum_{n=1}^\infty F(w_j,U_n\cap   B_{n})+F(w_j,T_{\delta_n}\setminus \overline B_{n})\nonumber\\
&\leq \sum_{n=1}^\infty F(u_{k_{n,j}}, B_{n}\setminus\overline  B_{n-1})+F(u_{k_{n+1,j}},B_{n+1}\setminus\overline  B_{n})+2r_{n,j}\nonumber\\
&\leq \frac2j+\liminf_{k\rightarrow \infty}F(u_k,A).
\end{align}
To conclude the proof, it is sufficient to observe that the sequence $w_j$ converges, as $j\rightarrow \infty$, to $u$ in $L^1(A;\R^m)$ thanks to \eqref{tesi1}; moreover, by \eqref{tesi2}, the previous convergence is strict in $BV(A;\R^m)$. Eventually, \eqref{tesi3} implies (ii), and the thesis is achieved.
%
%
%
%
%
%
%
%
%
%
%
%
\end{proof}
\begin{cor}\label{cor_main}
Assume the  hypotheses of Proposition \ref{prop31} and let $B_0=\varnothing$, and  $B_n$ ($n\geq1$) be the sets in Step 1 of its proof. Then the sequence $w_j$ also satisfies, for all $n\geq1$
\begin{align}\label{ineq37}
F(w_j,B_n\setminus B_{n-1})\leq \liminf_{k\rightarrow \infty}F(u_k,B_n\setminus B_{n-1})+\frac{1}{j2^{n+1}}.
\end{align}
If moreover $u_k$ is a recovery sequence for $\mathcal F(u,A)$, then $w_j$ is still a  recovery sequence for  $\mathcal F(u,A)$, $w_j\res (B_n\setminus \overline B_{n-1})$ is a recovery sequence for $\mathcal F(u,B_n\setminus \overline B_{n-1})$, and at the same time $w_j\res B_n$ is a recovery sequence for $\mathcal F(u,B_n)$ for all $n\geq1$.
\end{cor}
\begin{proof}
Inequality \eqref{ineq37} follows from the definition of $w_j$, expressions \eqref{glueingthesis2} and conditions (3) and (3*) in the proof of Proposition  \ref{prop31}. If $u_k$ is a recovery sequence for  $\mathcal F(u,A)$, then conditions (a), (b), and (c), in Step 1 of that proof ensure, thanks to Proposition \ref{prop:recoveryrestriction}, $u_k\res (B_n\setminus \overline B_{n-1})$ is a recovery sequence for $ \mathcal F(u,B_n\setminus \overline B_{n-1})$ and at the same time $u_k\res B_n$ is a recovery sequence for $\mathcal F(u,B_n)$; the thesis follows from (ii) of Proposition \ref{prop31} and \eqref{ineq37}.
\end{proof}

We are now in a position to check conditions (i)-(iv) of Theorem \ref{DGL_teo};  we start with the monotonicity condition (i):

\begin{theorem}(Monotonicity)
Let $B\subseteq A$ be bounded open sets  and let $u\in BV(A;\R^m)$; then
$$\mathcal F(u,B)\leq \mathcal F(u,A).$$ 
\end{theorem}
\begin{proof}
Let $(u_k)\subset\text{{\rm Lip}}_{\text{{\rm loc}}}(A;\R^m)$ be a recovery sequence for $\mathcal F(u,A)$. According to Proposition \ref{prop31} (applied to the case $A=\Om$ and $B$ in place of $A$) there exists a sequence $w_j \subset\text{{\rm Lip}}_{\text{{\rm loc}}}(B;\R^m)$ such that 
$$\liminf_{j\rightarrow \infty}F(w_j,B)\leq \liminf_{k\rightarrow \infty}F(u_k,B)\leq \liminf_{k\rightarrow \infty}F(u_k,A)=\mathcal F(u,A).$$
Since $\mathcal F(u,B)\leq \liminf_{j\rightarrow \infty}F(w_j,B)$ we have concluded.
\end{proof}

As additivity (ii) is trivial, we proceed to verify (iv) of Theorem \ref{DGL_teo}, and then go to (iii).

\begin{theorem}(Inner regularity)\label{thm_ir}
Let $A\subset\R^2$ be a bounded  open set; then 
\begin{align}
\mathcal F(u;A)=\sup\{\mathcal F(u;B):B\text{ is an open set and }B\subset\subset A\}.
\end{align}
\end{theorem}

\begin{proof}
 \textit{Step 1:}
	We consider the same setting in Step 1 of the proof of Proposition \ref{prop31}. In particular, we fix a recovery sequence $u_k$ for $\mathcal F(u,A)$, and assume that, for all $n\geq 1$, and $i=1,\dots,N_n$,
	\begin{itemize}
		\item[(a)] $|Du|(\Gamma_n^i)=0$ and $u\res \Gamma_n^i$ belongs to $BV(\Gamma_n^i;\R^m)$;
		\item[(b)] Setting $\psi_n^i(r):=|u\res ( \Gamma_n^i)_r|_{BV}=|D_\zeta u|((\Gamma_n^i)_r)$ then $0$ is a regular value for $\psi_n^i$;
		\item[(c)]  $u_k\res \Gamma_n^i\rightarrow u\res \Gamma_n^i$ strictly in $BV(\Gamma_n^i;\R^m)$.
	\end{itemize}
By standard arguments one sees that 
\begin{align}\label{sup=sup}
	\sup\{\mathcal F(u;B):B\text{ is an open set and }B\subset\subset A\}=
	\sup\{\mathcal F(u;B_n):n\geq 1\}.
\end{align}
Indeed,  let $B\subset\subset A$;
by compactness of $\overline B$ one has $\dist(B,A^c)>0$ and so there exists $n$ such that $B\subset A_{n}\subset B_{n}$. 
So, by monotonicity the inequality sign $\geq$ holds in \eqref{sup=sup}, and the converse being obvious, the claim follows.

We fix $\varepsilon>0$ arbitrary, and prove that there exists $n_\varepsilon$ such that  
\begin{align}
	\mathcal F(u,B_{n_\varepsilon})\geq\mathcal F(u,A)-\varepsilon.
\end{align}
This will imply the thesis by arbitrariness of $\varepsilon>0$.

\textit{Step 2:}	Condition (c) ensures that, thanks to Proposition \ref{prop:recoveryrestriction}, $u_k\res B_n$ and $u_k\res (B_{n+1}\setminus \overline B_n)$ are still recovery sequences for $\mathcal F(u,B_n)$ and $\mathcal F(u,B_{n+1}\setminus \overline B_n) $ respectively, for all $n\geq1$. This implies that
\begin{align*}
\mathcal F(u,B_n)=\lim_{k\rightarrow \infty}F(u_k,B_n)=\sum_{i=1}^n\lim_{k\rightarrow \infty}F(u_k,B_i\setminus\overline B_{i-1})=\sum_{i=1}^n\mathcal F(u,B_i\setminus\overline B_{i-1}),
\end{align*}	
	where once more we have set $B_0=\varnothing$.
	Since, by monotonicity, for all $n\geq 1$ we have $\mathcal F(u,B_n)\leq \mathcal F(u,A)$, we conclude
	\begin{align}\label{311}
\sum_{i=1}^\infty 	\mathcal F(u,B_i\setminus \overline B_{i-1})\leq\mathcal F(u,A). 
	\end{align}
%
%
Fix $\varepsilon>0$; by \eqref{311} the series in the left-hand side is convergent, and so we can fix $n_\varepsilon>0$ so that 
\begin{align}
	\sum_{i=n_\varepsilon+1}^\infty 	\mathcal F(u,B_i\setminus \overline B_{i-1})\leq\varepsilon,
\end{align}
 We consider the sequence $w_j$ provided by  Corollary \ref{cor_main}, that, for all $i\geq1$, is a recovery sequence for $\mathcal F(u,B_i\setminus \overline B_{i-1})$ and for $\mathcal F(u,B_{n_\varepsilon})$. 
From \eqref{ineq37} we deduce that 
\begin{align*}
	\mathcal F(u,A)&=\lim_{j\rightarrow \infty}F(w_j,B_{n_\varepsilon})+\lim_{j\rightarrow\infty}\sum_{i=n_\varepsilon+1}^\infty F(w_j,B_i\setminus \overline B_{i-1})\\
	&\leq \mathcal F(u,B_{n_\varepsilon})+\lim_{j\rightarrow\infty}\Big(\sum_{i=n_\varepsilon+1}^\infty\liminf_{k\rightarrow \infty}F(u_k,B_i\setminus\overline  B_{i-1})+\frac{1}{j2^{n+1}}\Big)\\
	&\leq \mathcal F(u,B_{n_\varepsilon})+\lim_{j\rightarrow\infty}\Big(\sum_{i=n_\varepsilon+1}^\infty	\mathcal F(u,B_i\setminus \overline B_{i-1})+\frac{1}{j}\Big)\\
	&= \mathcal F(u,B_{n_\varepsilon})+\varepsilon.
\end{align*}
By arbitrariness of $\varepsilon>0$ we conclude.
\end{proof}

\begin{theorem}(Sub-additivity)
	Let  $u\in BV(\Om;\R^m)$ be given. Then for all open sets $A_1,A_2,A\subset\Om$ with $A\subseteq A_1\cup A_2$ it holds
	\begin{align*}
\mathcal F(u,A)\leq \mathcal F(u,A_1)+\mathcal F(u,A_2).
	\end{align*} 
\end{theorem}

\begin{proof}
Let $u_k\subset{\text{{\rm Lip}}}_{\text{{\rm loc}}}(\Om;\R^m)$ be a recovery sequence for $\mathcal F(A_1\cup A_2)$. Starting from the set $A$, we build, as in Step 1 of the proof of Proposition \ref{prop31}, the sets $B_n\subset\subset A$, $n\geq1$. By definition
\begin{align}
B_n\subset A_{n+1}=\{x\in A:\dist(x,A^c)>\eta_{n+1}\}
\end{align}
and taking into account that $\partial B_n=\cup_{i=1}^{N_n}\Gamma_n^i$ enjoies properties (a), (b), and (c), we immediately obtain that $u_k\res B_n$ is a recovery sequence for $\mathcal F(u,B_n)$. Then we fix $\varepsilon>0$; owing to the inner regularity, Theorem \ref{thm_ir}, and thanks to \eqref{sup=sup}, we  choose  $n_\varepsilon>0$ so that 
\begin{align}\label{pi}
\mathcal F(u,B_{n_\varepsilon})\geq \mathcal F(u,A)-\varepsilon.
\end{align}
Next we proceed once again along the lines of Step 1 of Proposition \ref{prop31} for the sets $A_1$ and $A_2$, obtaining sets $B_n^1$ and $B_n^2$, $n\geq1$, for which
\begin{align*}
	&B^1_n\subset A^1_{n+1}=\{x\in A_1:\dist(x,A_1^c)>\eta^1_{n+1}\}\subset B^1_{n+1},\\
	&B^2_n\subset A^2_{n+1}=\{x\in A_2:\dist(x,A_2^c)>\eta^2_{n+1}\}\subset B^2_{n+1},	
\end{align*}
for suitable infinitesimal decreasing sequences of numbers $\eta_n^1$ and $\eta_n^2$ (which may differ from $\eta_n$). We therefore choose $\overline n$ big enough so that $\eta_{\overline n+1}^1,\eta_{\overline n+1}^2<\eta_{n_\varepsilon+1}$, and so we check that 
\begin{align}\label{315}
B_{n_\varepsilon}\subset A_{n_\varepsilon+1}\subset A^1_{\overline n+1}\cup A^2_{\overline n+1}\subset B^1_{\overline n+1}\cup B^2_{\overline n+1}.
\end{align}
Here the second inclusion is true since $A\subseteq A_1\cup A_2$, and so
\begin{align*}
&\{x\in A:\dist(x,A^c)>\eta_{n_\varepsilon+1}\}\subseteq  \{x\in A:\dist(x,(A_1\cup A_2)^c)>\eta_{n_\varepsilon+1}\}\\
&\qquad \subseteq  \{x\in A_1:\dist(x,(A_1\cup A_2)^c)>\eta_{n_\varepsilon+1}\}\cup  \{x\in A_2:\dist(x,(A_1\cup A_2)^c)>\eta_{n_\varepsilon+1}\};
\end{align*}
now since $\dist(x,(A_1\cup A_2)^c)=\min\{\dist(x,A_1^c),\dist(x,A_2^c)\}$, we also have
\begin{align*}
&\{x\in A_1:\dist(x,(A_1\cup A_2)^c)>\eta_{n_\varepsilon+1}\}\cup  \{x\in A_2:\dist(x,(A_1\cup A_2)^c)>\eta_{n_\varepsilon+1}\}\\
&\qquad \subseteq \{x\in A_1:\dist(x,A_1^c)>\eta_{n_\varepsilon+1}\}\cup  \{x\in A_2:\dist(x,A_2^c)>\eta_{n_\varepsilon+1}\}\\
&\qquad \subseteq \{x\in A_1:\dist(x,A_1^c)>\eta^1_{\overline n+1}\}\cup  \{x\in A_2:\dist(x,A_2^c)>\eta^2_{\overline n+1}\}= A^1_{\overline n+1}\cup A^2_{\overline n+1}.
\end{align*}
From \eqref{315} we can finally write, for all $k$,
\begin{align*}
F(u_k,B_{n_\varepsilon})\leq F(u_k,B^1_{\overline n+1})+F(u_k,B^2_{\overline n+1}),
\end{align*}
and so passing to the limit as $k\rightarrow \infty$ we end up to
\begin{align}
\mathcal F(u,B_{n_\varepsilon})\leq \mathcal F(u,B^1_{\overline n+1})+\mathcal F(u,B^2_{\overline n+1})\leq \mathcal F(u,A_1)+\mathcal F(u,A_2),
\end{align}
the second iequality following from monotonicity of $\mathcal F(u,\cdot)$.
This implies the thesis thanks to 
\eqref{pi} and the arbitrariness of $\varepsilon$.
\end{proof}

\section{Examples of representation formulas}\label{sec_representations}

In this section we revise some examples showing how the  area functional relaxed with respect to strict topology is representable in an integral form.

Consider a rectangle $R:=(a,b)\times (c,d)\subset\R^2$, let $h\in (c,d)$ and let $S:=(a,b)\times h$. Let $R^+:=(a,b)\times (h,d)$, $R^-:=(a,b)\times (c,h)$, and $u\in BV(R;\R^2)$ be a map such that $u^\pm:=u\res R^\pm$ are Lipschitz continuous. In this case the relaxed area $\mathcal A(u,R)$ has been proved to be \cite{BCS2}
\begin{align}
\mathcal A(u,R):=\mathbb A(u,R^+)+\mathbb A(u,R^-)+\int_{(a,b)\times (0,1)}|\partial_t X^{\text{{\rm aff}}}\wedge \partial_s X^{\text{{\rm aff}}}|dtds,
\end{align}  
where $X^{\text{{\rm aff}}}$ is the affine interpolation between the traces of $u^\pm$ on $S$, namely
\begin{align}
X^{\text{{\rm aff}}}(t,s):=(t,s u^+(t,h)+(1-s)u^-(t,h)),\qquad \qquad \forall (t,s)\in (a,b)\times (0,1).
\end{align}
This result can be extended to piecewise Lipschitz maps with jump forming a network (namely a graph consisting of finitely many $C^2$-curves meeting at finitely many junctions points, see \cite{BCS2}). A similar representation formula holds for this kind of maps, where however there appears also the singular contribution of solutions of suitable plateau problems accounting for the junctions points (see \cite[Theorem 1.1]{BCS2}).

Another important case is the one of Sobolev maps with values in $\mathbb S^1$, $u\in W^{1,1}(\Om;\mathbb S^1)$. In this case, if $\det(\nabla u)=\pi \sum_{i=1}^\infty (\delta_{x_i}-\delta_{y_i})$ (see \cite{BSS} and references therin), then the measure $\mu(A):= \mathcal A(u,A)$ takes the form
$$\mu=\sqrt{1+|\nabla u|^2}\cdot\mathcal L^2+\pi\sum_{i=1}^\infty(\delta_{x_i}+\delta_{y_i}).$$ 
For general maps of bounded variation $u$ an explicit expression of $\mu$ is not known at the present stage. This will be object of future research.

\subsection{A Cartesian map with singular relaxed area}

We consider a Lipschitz curve $\varphi:\mathbb S^1\rightarrow \R^2$ and, for $\Om=B_r$, $r>0$, the $0$-homogeneous map $u_\varphi:\Om\subset\R^2\rightarrow \R^2$ given by
\begin{align}\label{omomap}
u_\varphi(x)=\varphi(\frac{x}{|x|}),\qquad \qquad x\in \Om\setminus \{0\}.
\end{align}
It is easy to see that the graph of $u_\varphi$, treated as a $2$-integral current  $\mathcal G_{u_\varphi}\in \mathcal D_2(\Om\times \R^2)$, satisfies
\begin{align*}
\partial {\mathcal G}_{u_\varphi}=\delta_0\times \varphi_\sharp\jump{\mathbb S^1}\qquad \qquad \text{ in }\mathcal D_1(\Om\times\R^2),
\end{align*}
where $\varphi_\sharp\jump{\mathbb S^1}$ is the integration over the image of $\varphi$, i.e., the push-forward by $\varphi$ of the standard integration over the unit circle $\mathbb S^1$.
According to the results in \cite{BCS} (see also \cite{C}) it holds
\begin{align}\label{sing_plateau}
\mathcal A(u_\varphi,\Om)=\int_\Om \sqrt{1+|\nabla u_\varphi|^2}dx+\mathcal P(\varphi),
\end{align}
where $\mathcal P(\varphi)$ corresponds to the area of a disk-type solution of the planar Plateau problem with boundary $\varphi(\mathbb S^1)$. Specifically
\begin{align}
\mathcal P(\varphi):=\inf\{\int_{B_1}|\partial_{x_1}\Phi\wedge \partial_{x_2}\Phi| dx:\Phi=\varphi\text{ on }\partial B_1,\;\Phi\in \text{{\rm Lip}}(B_1;\R^2)\}.
\end{align}
This Plateau problem can be singular, in the sense that the contour $\varphi(\mathbb S^1)$ of the minimal disk can have self-intersection and overlappings (see \cite{Hass:91,CF,Creutz:20,CS} for this kind of Plateau problem and generalization). It is interesting to observe that this singular contribution is related with the presence of the Jacobian determinant in the integrand of our functional. Indeed, a similar contribution appears when we consider the total variation of the Jacobian (see \cite{BCS,C}), relaxation with respect to the strict convergence in $BV$ of \eqref{TVJ}:
\begin{align}
\mathcal{TVJ}(u_\varphi,\Om)=\mathcal P(\varphi),
\end{align}
 (compare with the results in \cite{Paol} and \cite{DP}).

We now make a specific choice for $\varphi$: Let $\Gamma_1$ and $\Gamma_2$ be two circumferences tangent to each other at the origin $0$. If $\alpha_i$ denotes a constant speed parametrization of $\Gamma_i$ starting from $0$, we consider the concatenation 
\begin{align}\label{otto}
\varphi:=\alpha_1\star \alpha_2\star \alpha_1^{-1}\star \alpha_2^{-1},
\end{align}
that is a Lipschitz closed curve running the $8$-shaped  figure consisting of $\Gamma_1\cup \Gamma_2$ two times, the first with opposite orientation of the second time. Due to this, it turns out that the current $\varphi_\sharp\jump{\mathbb S^1}$ is null, so that $u_\varphi$ is a Cartesian map, namely
\begin{align*}
	\partial \mathcal G_{u_\varphi}=0\qquad \qquad \text{ in }\mathcal D_1(\Om\times\R^2).
\end{align*}
At the same time \eqref{sing_plateau} still holds, and $\mathcal P(\varphi)$ is nonzero; indeed it turns out that  $\mathcal P(\varphi)$ coincides with two times the area of the smaller circle between $\Gamma_1$ and $\Gamma_2$ (see \cite{Paol,CS}).

We now prove the following interesting observation:

\begin{theorem}\label{teo_main2}
Let $r>0$ and $u_\varphi:B_r(0)\rightarrow\R^2$ the Cartesian map in \eqref{omomap} with $\varphi$ be the double eight curve in \eqref{otto}. Then, it holds
\begin{align}
\mathcal A^{L^1}(u_\varphi,B_r)>\int_{B_r} \sqrt{1+|\nabla u_\varphi|^2}dx.
\end{align}
\end{theorem}
 In other words we have found a Cartesian map whose area functional, even if relaxed with respect to the $L^1$-topology, is strictly greater than the area of its graph.
 
 \begin{proof}
Assume by contradition that for some $\overline r>0$ it holds
$$\mathcal A^{L^1}(u_\varphi,B_{\overline r})=\int_{B_{\overline r}} \sqrt{1+|\nabla u_\varphi|^2}dx.$$
Let $(u_k)\subset C^1(B_{\overline r};\R^2)$ be a recovery sequence for $\mathcal A^{L^1}(u_\varphi,B_{\overline r})$ and denote $V_k:=\nabla u_k$. We have
\begin{align*}
\limsup_{k\rightarrow \infty}\int_{B_{\overline r}}\sqrt{1+|V_k|^2}dx\leq \lim_{k\rightarrow \infty}\int_{B_{\overline r}}\sqrt{1+|V_k|^2+|\det(\nabla u_k)|^2}=\int_{B_{\overline r}} \sqrt{1+|\nabla u_\varphi|^2}dx
\end{align*}
and, on the other hand, by lower semicontinuty 
\begin{align*}
\liminf_{k\rightarrow \infty}\int_{B_{\overline r}}\sqrt{1+|V_k|^2}dx\geq \int_{B_{\overline r}} \sqrt{1+|\nabla u_\varphi|^2}dx.
\end{align*}
So $\lim_{k\rightarrow \infty}\int_{B_{\overline r}}\sqrt{1+|V_k|^2}dx= \int_{B_{\overline r}} \sqrt{1+|\nabla u_\varphi|^2}dx$; hence by Proposition \ref{prop34} we conclude
$V_k=\nabla u_k\rightarrow \nabla u_\varphi$ strongly in $L^1(B_{\overline r})$. But strong convergence of gradients implies strict convergence in $BV(B_{\overline r};\R^2)$, so by \eqref{sing_plateau} we arrive at
\begin{align*}
\liminf_{k\rightarrow \infty}\mathbb A(u_k,B_{\overline r})\geq \mathcal A(u_\varphi,B_{\overline r}
)=\int_{B_{\overline r}} \sqrt{1+|\nabla u_\varphi|^2}dx+\mathcal P(\varphi)>\int_{B_{\overline r}} \sqrt{1+|\nabla u_\varphi|^2}dx,
\end{align*}
 a contradiction.
\end{proof} 

\section{Appendix}
We collect here some useful results for the above discussion.
\begin{lemma}\label{Fatou=}
	Let $A\subset\R$ be a bounded open set and let $f_k,f\in L^1(A)$ be non-negative functions such that
	$$\lim_{k\rightarrow \infty}\int_Af_kdx=\int_Afdx,\qquad \qquad f(x)=\liminf_{k\rightarrow \infty}f_k(x)\qquad \text{a.e. }x\in A.$$
	 Then $f_k\rightarrow f$ in $L^1(A)$.	
\end{lemma}

\begin{proof}
We prove that $\psi_k:=f_k-f$ tends to $0$ in $L^1(A)$. To this aim, we denote by $\psi_k^+=\psi_k\vee 0$ and $\psi_k^-=(-\psi_k)\vee 0$ the positive and negative parts of $\psi_k$, respectively, so that it is enough to show that they both tends to $0$ in $L^1(A)$. As for the negative part, we readily see that  $\psi_k^-=(f-f_k)\vee0\leq f$, and moreover from $f(x)=\liminf_{k\rightarrow \infty}f_k(x)$ we deduce that $\limsup_{k\rightarrow \infty}f(x)-f_k(x)=0$, hence $\lim_{k\rightarrow \infty}\psi_k^-=0$ a.e. on $A$. Therefore, by Dominated Convergence Theorem $\psi_k^-\rightarrow 0$ in $L^1(A)$. 

This also allows to treat the  positive part, since we know that $0=\lim_{k\rightarrow \infty}\int_{A}\psi_kdx=\lim_{k\rightarrow \infty}\int_{A}\psi^+_kdx$, which implies $\psi_k^+\rightarrow 0$ in $L^1(A)$. The thesis is achieved. 
\end{proof}

The following result can be found in \cite{Fu}:

\begin{lemma}\label{lem_Fu}
	Let $U\subset\R^2$ be a relatively compact set; then for a.e. $t>0$ the set 
	$$\Gamma_t:=\{x\in \R^2:\dist(x,U)=t\},$$
	consists of finitely many Lipschitz curve.
\end{lemma}

\begin{proof}
This follows from the fact that for a.e. $t$ the set $U_t:=\{x\in \R^2:\dist(x,U)<t\}$ is an open set with Lipschitz boundary.
\end{proof}

\textbf{Acknowledgements:}
The author is  member of the Gruppo Nazionale per l'Analisi Matematica, la Probabilit\`a e le loro Applicazioni (GNAMPA) of the Istituto Nazionale di Alta Matematica (INdAM), and joins the project CUP\_E53C22001930001. He also acknowledges  the financial support of 
PRIN 2022PJ9EFL "Geometric Measure Theory: Structure of Singular Measures, Regularity Theory and Applications in the Calculus of Variations". The latter has been funded by the European Union under NextGenerationEU, project number CUP B53D23009400006. Views and opinions expressed are however those of the author(s) only and do not necessarily reflect those of the European Union or The European Research Executive Agency.  Neither the European Union nor the granting authority can be held responsible for them.

\end{document}

\end{document}